\documentclass[12pt]{article}
\usepackage[utf8]{inputenc}
\usepackage{amssymb,amsmath,amsfonts,eucal,mathrsfs,amsthm} 
\usepackage[colorinlistoftodos]{todonotes}
\usepackage[allcolors=blue]{hyperref}
\usepackage{empheq}
\usepackage{bm}
\newtheorem{theorem}{Theorem}
\newtheorem{proposition}[theorem]{Proposition}
\newtheorem{lemma}[theorem]{Lemma}

\newtheorem{corollary}[theorem]{Corollary}

\newtheorem{remark}[theorem]{Remark}

\newtheorem{example}[theorem]{Example}
\newtheorem{examples}[theorem]{Examples}
\newtheorem{definition}[theorem]{Definition}
\theoremstyle{definition}
\newcommand{\R}{\mathbb{R}}
\newcommand{\Q}{\mathbb{Q}}
\newcommand{\Sf}{\mathbb{S}}

\newcommand{\Hy}{\mathbb{H}}
\newcommand{\spa}{\mbox{span}}

\newcommand{\rank}{\mbox{rank }}

\newcommand{\grad}{\mbox{grad\,}}

\newcommand{\nab}{\tilde\nabla}

\newcommand{\End}{\mbox{End}}
\newcommand{\Hom}{\mbox{Hom}}

\newcommand{\Les}{\mathbb{L}}
\def\<{{\langle}}
\def\>{{\rangle}}

\def\a{\alpha}

\def\be{\begin{equation} }
\def\ee{\end{equation} }

\def\proof{\noindent{\it Proof:  }}
\def\qed{\ifhmode\unskip\nobreak\fi\ifmmode\ifinner
\else\hskip5 pt \fi\fi\hbox{\hskip5 pt \vrule width4 pt
height6 pt  depth1.5 pt \hskip 1pt }}
\makeatletter
\newcommand{\subjclass}[2][]{\let\@oldtitle\@title
\gdef\@title{\@oldtitle\footnotetext{#1 
\emph{Mathematics Subject Classification:} #2}}}
\newcommand{\keywords}[1]{\let\@@oldtitle\@title
\gdef\@title{\@@oldtitle\footnotetext
{\emph{Key words and phrases.} #1.}}}
\makeatother

\title{}
\author{}
\date{}

\begin{document}
\title{Deformable hypersurfaces of $\mathbb{H}^k\times \mathbb{S}^{n-k+1}$}
\maketitle
\begin{center}
\author{M. I. Jimenez        and
        R. Tojeiro$^*$}  
        \footnote{Corresponding author}
      \footnote{This research was initiated while the first author was supported by CAPES-PNPD Grant 88887.469213/2019-00 and was finished under the support of Fapesp Grant 2022/05321-9. The second author was partially supported by Fapesp grant 2022/16097-2 and CNPq grant 307016/2021-8.\\
      Data availability statement: Not applicable.}
\end{center}
\date{}

\begin{abstract}

We investigate the deformable hypersurfaces of the conformally flat Riemannian products $\mathbb{H}^k\times \mathbb{S}^{n-k+1}$, $2\leq k\leq n-1$,  of hyperbolic space and a sphere with constant sectional curvatures $-1$ and $1$, respectively.
   
\end{abstract}

\vspace{2ex}


\date{}
\maketitle

\section{Introduction}

  A basic problem in Submanifold theory is to investigate the uniqueness of an isometric immersion $f\colon M^n\to \tilde{M}^{n+p}$ of a Riemannian manifold $M^n$ of dimension $n$ into another  Riemannian manifold $\tilde{M}^{n+p}$ of dimension $n+p$. Clearly, since $\tau\circ f$ is also an isometric immersion for any isometry  $\tau\colon \tilde{M}^{n+p}\to \tilde{M}^{n+p}$, uniqueness should be understood up to isometries of $\tilde{M}^{n+p}$. The usual terminology for uniqueness in this sense is \emph{isometric rigidity}, or simply, \emph{rigidity}. If an isometric immersion $f\colon M^n\to \tilde{M}^{n+p}$ is not rigid on any open subset of $M^n$, it is said to be \emph{isometrically deformable}, or simply, \emph{deformable}, and any isometric immersion $g\colon M^n\to \tilde{M}^{n+p}$ that is not given as $g=\tau\circ f$ for some isometry  $\tau\colon \tilde{M}^{n+p}\to \tilde{M}^{n+p}$ on any open subset of $M^n$ is said to be an \emph{isometric deformation} of $f$.
  
  Until now, the preceding \emph{isometric deformation problem} for isometric immersions $f\colon M^n\to \tilde{M}^{n+p}$, $n\geq 3$, has been studied only when the ambient manifold $\tilde{M}^{n+p}$ is a space form, starting with the pioneering works of Sbrana \cite{sb} and Cartan \cite{ca} for the case in which $p=1$ and the ambient manifold is Euclidean space $\mathbb{R}^{n+1}$; the case in which the ambient manifold is the sphere and the hyperbolic space was addressed in \cite{DFT2} (see \cite{DT} for an account of several known results on the isometric deformation problem for submanifolds of space forms; see also \cite{Gr} and \cite{Je} for an approach to the problem based on Cartan's method of moving frames). This is mostly because a ``Fundamental theorem of submanifolds" was only available until recently for submanifolds of space forms. After such a theorem was established, in particular, for submanifolds with arbitrary dimension and codimension of a product of space forms in \cite{LTV} and \cite{Ko}, extending previous results in \cite{Da} for hypersurfaces of $\mathbb{S}^n\times \mathbb{R}$ and $\mathbb{H}^n\times \mathbb{R}$, it became a natural problem to study the isometric deformation problem in these spaces. 
  
  In this article, we initiate the study of deformable hypersurfaces of products of space forms. More specifically,
  we investigate the deformable hypersurfaces of the conformally flat Riemannian products $\mathbb{H}^k\times \mathbb{S}^{n-k+1}$, $2\leq k\leq n-1$,  of hyperbolic space and a sphere with constant sectional curvatures $-1$ and $1$, respectively.
  
  In terms of the model of $\mathbb{H}^k\times \mathbb{S}^{n-k+1}$ as a submanifold of $\mathbb{L}^{k+1}\times \mathbb{R}^{n-k+2}=\mathbb{L}^{n+3}$, where $\mathbb{L}^{m}$ stands for Lorentz space of dimension $m$, a class of examples of deformable hypersurfaces arises by considering pairs of hypersurfaces   $f, g\colon M^n\to \mathbb{H}^k\times \mathbb{S}^{n-k+1}$ that are \emph{weakly congruent}. This means that  $g=\Phi \circ f$ for some linear isometry   $\Phi\colon \mathbb{L}^{n+3}\to \mathbb{L}^{n+3}$ that does not preserve $\mathbb{L}^{k+1}$ and $\mathbb{R}^{n-k+2}$, and therefore does not restrict itself to an isometry of $\mathbb{H}^k\times \mathbb{S}^{n-k+1}$. Then $f$ is a deformable hypersurface, with $g$ an isometric deformation of $f$ (and conversely), which we call a \emph{weakly congruent deformation} of $f$.

     We say that a hypersurface $f\colon M^n\to \mathbb{H}^k\times \mathbb{S}^{n-k+1}$ is \emph{truly deformable} if it admits an isometric deformation $g\colon M^n\to \mathbb{H}^k\times \mathbb{S}^{n-k+1}$ that is not weakly congruent to $f$ in any open subset of $M^n$.  Any such isometric deformation is said to be a \emph{true deformation} of $f$.    The simplest examples of truly deformable hypersurfaces of $\mathbb{H}^k\times \mathbb{S}^{n-k+1}$ are products $f=f_1\times i_2\colon N^{k-1}\times\Sf^{n-k+1}\to \Hy^{k}\times\Sf^{n-k+1}$ or $f=i_1\times f_2\colon \Hy^{k}\times N^{n-k}\to \Hy^{k}\times\Sf^{n-k+1}$, where $f_1\colon N^{k-1}\to \Hy^k$ and $f_2\colon N^{n-k}\to \Sf^{n-k+1}$ are deformable hypersurfaces of $\Hy^{k}$ and $\Sf^{n-k+1}$, respectively, and  $i_1$ and $i_2$ are the respective identity maps. 
  For $f=f_1\times i_2\colon N^{k-1}\times\Sf^{n-k+1}\to \Hy^{k}\times\Sf^{n-k+1}$, for instance, if $g_1\colon N^{k-1}\to \Hy^k$ is an isometric deformation of $f_1$, then $g=g_1\times i_2$ is a true deformation of $f$. Similarly for $f=i_1\times f_2\colon \Hy^{k}\times N^{n-k}\to \Hy^{k}\times\Sf^{n-k+1}$.

  There exists an explicit conformal diffeomorphism $\Psi$ between $\Hy^{k}\times\Sf^{n-k+1}$ and the complement $\R^{n+1}\setminus\R^{k-1}$ of a subspace
  $\R^{k-1}\subset \R^{n+1}$. Hence, a pair of isometric hypersurfaces in $\Hy^{k}\times\Sf^{n-k+1}$ gives rise to a pair of conformal hypersurfaces in $(\R^{n+1}\setminus\R^{k-1}, g_0)$, where $g_0$ is the Euclidean metric. The isometries of $(\R^{n+1}\setminus\R^{k-1}, (\Psi^{-1})^*g)$, where $g$ is the metric of $\Hy^{k}\times\Sf^{n-k+1}$, are the Moebius transformations of $\R^{n+1}$ that leave $\R^{k-1}$ invariant (see, e.g., Theorem $1$ of \cite{JT}). 
  Therefore, two hypersurfaces of the conformal model $(\R^{n+1}\setminus\R^{k-1}, (\Psi^{-1})^*g)$ of $\Hy^{k}\times\Sf^{n-k+1}$  are congruent if and only if they differ by a Moebius transformation of $\R^{n+1}$ that leaves $\R^{k-1}$ invariant. They are weakly congruent deformations of one another if and only if they differ by an arbitrary Moebius transformation of $\R^{n+1}$.
  On the other hand, the  truly deformable hypersurfaces of $\Hy^{k}\times\Sf^{n-k+1}$, when regarded as hypersurfaces of $(\R^{n+1}\setminus\R^{k-1}, (\Psi^{-1})^*g)$,  are among the conformally deformable hypersurfaces of $\R^{n+1}$ classified by Cartan in \cite{ca2} (see also \cite{DT1} and Chapter $17$ of~\cite{DT}).  These are the hypersurfaces $f\colon M^n\to\R^{n+1}$ that admit a nontrivial \emph{conformal deformation} $\tilde f\colon M^n\to\R^{n+1}$, that is,  an immersion such that $f$ and $\tilde f$ induce conformal metrics on $M^n$ and do not differ by a Moebius transformation of $\R^{n+1}$ on any open subset of $M^n$.

According to Cartan's classification, besides conformally flat hypersurfaces, which have a principal curvature with multiplicity greater than or equal to $n-1$ and are highly conformally deformable, the remaining ones fall into one of the following classes:
\begin{itemize}
\item[(i)] \emph{conformally surface-like hypersurfaces}, that is, those that differ by a Moebius transformation of $\mathbb{R}^{n+1}$ from cylinders and rotational hypersurfaces over surfaces in $\R^3$, or from cylinders over three-dimensional 
hypersurfaces of $\R^4$ that are cones over surfaces in $\Sf^3$; 
\item[(ii)]  \emph{conformally ruled hypersurfaces}, that is, hypersurfaces $f\colon M^n\to\R^{n+1}$ for which $M^n$ carries an integrable  distribution of rank $(n-1)$
whose leaves are mapped by $f$ into umbilical submanifolds of $\R^{n+1}$;
\item[(iii)] hypersurfaces that admit a  non-trivial \emph{conformal variation}  $F\colon (-\epsilon,\epsilon)\times M^n \to \mathbb{R}^{n+1}$, that is,  a smooth map 
 such that, for any $t \in (-\epsilon,\epsilon)$, the map  $f_t = F(t; \cdot)$,  with $f_0 = f$, is a non-trivial \emph{conformal deformation} of $f$;
\item[(iv)]  hypersurfaces that admit a single non-trivial conformal deformation.
\end{itemize}

   Our main result shows that, except for the simple examples of truly isometrically deformable hypersurfaces of $\Hy^{k}\times\Sf^{n-k+1}$ mentioned above, all truly deformable hypersurfaces of dimension $n\geq 6$ of $\mathbb{H}^k\times \mathbb{S}^{n-k+1}$, $2\leq k\leq n-1$, when regarded as hypersurfaces of the conformal model $(\R^{n+1}\setminus\R^{k-1}, (\Psi^{-1})^*g)$  of $\Hy^{k}\times\Sf^{n-k+1}$, are either some special conformally flat hypersurfaces or particular examples in the first two families in Cartan's classification. 
   
      Among them are the ruled and conformally ruled hypersurfaces. We say that a hypersurface $f\colon M^n\to \mathbb{H}^k\times \mathbb{S}^{n-k+1}$ is \emph{ruled} (respectively, \emph{conformally ruled}) if $M^n$ admits a totally geodesic (respectively, umbilical) distribution of rank $(n-1)$ such that the restriction of $f$ to each of its leaves, called \emph{rulings}, is totally geodesic (respectively, umbilical). 
      
   All remaining truly deformable hypersurfaces of $\mathbb{H}^k\times \mathbb{S}^{n-k+1}$ are certain rotational or doubly rotational hypersurfaces defined below.
 
  First, let $\Psi$ be an isometry of a warped product $(\mathbb{H}^{k}\times \mathbb{S}^{n-k-s+1})\times_{\sigma} \mathbb{S}^{s}$
   onto an open subset of $\Hy^k\times \Sf^{n-k+1}$, called a \emph{warped product representation of $\Hy^k\times \Sf^{n-k+1}$} (see Section \ref{Sec:Rotational}). Let
 $f_0\colon M_0^{n-s}\to \mathbb{H}^{k}\times \mathbb{S}^{n-k-s+1}$ be a hypersurface and let $i_2\colon \mathbb{S}^{s}\to \mathbb{S}^{s}$ be the identity map.  
 Then the map 
 $$
 f\colon M_0^{n-s}\times_{\rho} \mathbb{S}^{s}\to \Hy^k\times \Sf^{n-k+1},\,\,\,\,\,f=\Psi\circ (f_0\times  i_2),
 $$
 where  $\rho=\sigma\circ f_0$,  is an isometric immersion that we call a \emph{rotational hypersurface of type $I$ with $f_0$ as profile}.
 
 Similarly, let $\Psi$ be an isometry of $(\mathbb{H}^{k-r}\times \mathbb{S}^{n-k+1})\times_{\sigma} \mathbb{Q}_\epsilon^{r}$
   onto an open subset of $\Hy^k\times \Sf^{n-k+1}$. Here, $\mathbb{Q}_\epsilon^{r}$ denotes $\mathbb{S}^r$, $\mathbb{R}^r$ or $\mathbb{H}^r$, according to whether $\epsilon =1$, $0$ or $-1$, respectively.
   Let $f_0\colon M_0^{n-r}\to \mathbb{H}^{k-r}\times \mathbb{S}^{n-k+1}$ be a hypersurface and let $i_1\colon \mathbb{Q}_\epsilon^{r}\to \mathbb{Q}_\epsilon^{r}$ be the identity map.
   Then the map 
   $$
 f\colon M_0^{n-r}\times_\rho\mathbb{H}^{r+s}\to \Hy^k\times \Sf^{n-k+1},\,\,\,\,\,f=\Psi\circ (f_0 \times i_1),
 $$
 where  $\rho=\sigma\circ f_0$, is also an isometric immersion,  called a \emph{rotational hypersurface of type $(II)$ with $f_0$ as profile}. 
 
Finally, let $\Psi$ be an isometry of $(\mathbb{H}^{k-r}\times \mathbb{S}^{n-k-s+1})\times_{\sigma_1} \mathbb{Q}_\epsilon^{r}\times_{\sigma_2} \mathbb{S}^{s}$
   onto an open subset of $\Hy^k\times \Sf^{n-k+1}$. Let
 $f_0\colon M_0^{n-r-s}\to \mathbb{H}^{k-r}\times \mathbb{S}^{n-k-s+ 1}$ be a hypersurface and let $i_1\colon \mathbb{Q}_\epsilon^{r}\to \mathbb{Q}_\epsilon^{r}$ and $i_2\colon \mathbb{S}^{s}\to \mathbb{S}^{s}$ be the identity maps.  Then the map 
 $$f\colon M_0^{n-r-s}\times_{\rho_1}\mathbb{Q}_\epsilon^{r}\times_{\rho_2}\Sf^s\to \Hy^k\times \Sf^{n-k+1},\,\,\,\,\,f=\Psi\circ (f_0 \times i_1\times i_2),$$
     where  $\rho_i=\sigma\circ f_0$ for $i=1,2$, is an isometric immersion that we call a \emph{doubly rotational hypersurface with $f_0$ as profile. }
     \vspace{1ex}

   We can now give a precise statement of the main result of this article.

   \begin{theorem}\label{main}
Let $f\colon M^n\to \Hy^k\times \Sf^{n-k+1}$, $n\geq 6$, $2\leq k\leq n-1$, be a truly deformable hypersurface. Then there exists an open and dense subset of $M^n$ where one of the following holds locally 
\begin{itemize}
\item[(i)] $f=f_1\times i_2\colon N^{k-1}\times\Sf^{n-k+1}\to \Hy^{k}\times\Sf^{n-k+1}$, where $f_1\colon N^{k-1}\to \Hy^k$ is a deformable hypersurface and  $i_2$ is the  identity map;
\item[(ii)] $f=i_1\times f_2\colon \Hy^{k}\times N^{n-k}\to \Hy^{k}\times\Sf^{n-k+1}$, where $f_2\colon N^{n-k}\to \Sf^{n-k+1}$ is a deformable hypersurface and  $i_1$ is the identity map;
\item[(iii)] $f$ is ruled;
\item[(iv)]  $f$ is conformally ruled;
\item[(v)]  $k=2$ (respectively, $k=n-1$) and $f$ is a rotational hypersurface  of type $(I)$ (respectively, type $(II)$),  with a surface  $f_0\colon M_0^2\to \mathbb{H}^{2}\times \mathbb{S}^{1}$  (respectively, $f_0\colon M_0^2\to \mathbb{H}^{1}\times \mathbb{S}^{2}$) as profile;
\item[(vi )] $2\leq k\leq 3$ (respectively, $n-2\leq k\leq n-1$) and $f$ is a rotational hypersurface  of type $(I)$ (respectively, type $(II)$) with a three-dimensional hypersurface $f_0\colon M_0^3\to \mathbb{H}^\ell\times \mathbb{S}^{4-\ell}$, $\ell\in \{2,3\}$ (respectively, $\ell\in \{1,2\}$), as profile;
\item[(vii)] $f$ is  a doubly rotational hypersurface whose profile is a surface $f_0\colon M_0^2\to \mathbb{H}^\ell\times \mathbb{S}^{3-\ell}$, $\ell\in \{1,2\}$.
\item[(viii)] $f$ is a doubly rotational hypersurface whose profile is a three-dimensional hypersurface $f_0\colon M_0^3\to \mathbb{H}^\ell\times \mathbb{S}^{4-\ell}$, $\ell\in \{1,2,3\}$.
\end{itemize}
\end{theorem}

 The paper is organized as follows. In the next section, we recall from \cite{LTV} the basic equations of a hypersurface $f\colon\,M^n\to \Q_{c_1}^{k}\times \Q_{c_2}^{n-k+1}$ of a product of space forms with constant sectional curvatures $c_1$ and $c_2$, respectively. These are given in terms of its shape operator $A$, a symmetric endomorphism $R$, and a one-form $S$ on $M^n$. The data $(A,R,S)$ uniquely determine the hypersurface $f$ up to isometries of $\Q_{c_1}^{k}\times \Q_{c_2}^{n-k+1}$ (see Theorem \ref{exist} and Proposition~\ref{cong}). 
 
 In Section~\ref{Sec:Rotational} we define rotational and doubly rotational hypersurfaces of $\mathbb{H}^k\times \mathbb{S}^{n-k+1}$ and derive a N\"olker-type theorem characterizing them as hypersurfaces that are intrinsically warped products and satisfy some additional conditions (see Proposition \ref{propmulti}). 
 
 The notion of weak congruence for a pair of hypersurfaces of $\mathbb{H}^k\times \mathbb{S}^{n-k+1}$ is introduced in Section~\ref{Sec:weak},  and it is shown how the data associated with weakly congruent hypersurfaces are related (see Proposition \ref{congL}).

In Section \ref{Sec:flatbil}, we show that a pair of isometric immersions $f,g\colon M^n\to \mathbb{H}^k\times \mathbb{S}^{n-k+1}$, $2\leq k\leq n-1$, gives rise to a certain flat bilinear form $\beta$ (see Proposition \ref{prop:flat}). Here is where the assumptions that $k\geq 2$ and $k\leq n-1$ are needed: a different notion of weak congruence appears when $k=1$ or $k=n-1$. In those cases, a different flat bilinear form has to be used (the case $k=1$ is considered in a forthcoming work with S. Pinillos). Weak congruence of $f$ and $g$ is characterized in terms of $\beta$ (see Proposition~\ref{prop:null}).

In Section \ref{Sec:weakdef}, we show how to produce hypersurfaces in $\mathbb{H}^k\times \mathbb{S}^{n-k+1}$ that admit weakly congruent deformations (see Proposition~\ref{dif}). We prove that these include all umbilical nontotally geodesic hypersurfaces (see Proposition~\ref{defumb}), but we also construct an explicit two-parameter family of non-umbilical hypersurfaces that admit weakly congruent deformations (see Example \ref{nonumbilical}). 

Section~\ref{Sec:truly} is devoted to the investigation of the truly deformable hypersurfaces of $\mathbb{H}^k\times \mathbb{S}^{n-k+1}$, $2\leq k\leq n-1$. As a key step, using the results of Section~\ref{Sec:flatbil}, we first obtain a necessary condition for a hypersurface of $\mathbb{H}^k\times \mathbb{S}^{n-k+1}$, $n\geq 5$, $2\leq k\leq n-1$, to admit true deformations (see Proposition \ref{commoneig}). 
This already implies a Beez-Killing-Cartan-type theorem, providing a sufficient condition for a hypersurface of $\mathbb{H}^k\times \mathbb{S}^{n-k+1}$, $n\geq 5$, $2\leq k\leq n-1$, to admit no true deformations (see Corollary \ref{beez}). 
Then we show that a stronger necessary condition must be satisfied when $n\geq 6$.
Namely, such a hypersurface and a true deformation of it must share a common eigendistribution $\Delta$ of rank at least $n-2$ associated with a common principal curvature $\lambda$, restricted to which the respective endomorphisms $R$ coincide (see Proposition~\ref{commoneig2}). This leads to a case-by-case study according to the possible values of the rank of $\Delta$ and depending on whether $\lambda$ is identically zero or vanishes nowhere. The proof of Theorem~\ref{main} then follows by putting together the conclusions drawn in each possible case. 

Finally, in Section~\ref{Sec:final} we prove some partial results and pose some problems regarding the converse of Theorem \ref{main}.

\section{Preliminaries}

The basic equations of a submanifold with arbitrary dimension and codimension of a product of two space forms have been established in  \cite{LTV} (see also \cite{Ko}). In this section, we rewrite those equations in the special case of hypersurfaces, in which we are interested in this article.

 Let  $f\colon\,M^n\to \Q_{c_1}^{k}\times \Q_{c_2}^{n-k+1}$ be  an isometric immersion of a Riemannian manifold.
  Let $N$ be a (local) unit normal vector field along $f$, and let $A=A^f_N\in \Gamma(T^*M\otimes TM)$ be the shape operator of $f$ with respect to $N$. Let $\xi=\xi_f\in \mathfrak{X}(M)$, $t=t_f\in C^{\infty}(M)$ and $R=R_f\in \Gamma(\End(TM))$ be defined by
$$\pi_2N=f_*\xi+tN$$ 
and 
\begin{equation}\label{L}
\pi_2 f_*X=f_*RX+\<X, \xi\>N,
\end{equation}
where $\pi_2$ denotes both the projection onto
 $\Q_{c_2}^{n-k+1}$ and its derivative. From
$\pi_2^2=\pi_2$
it follows that
\be\label{eq:pi1}
R(I-R)X=\<X, \xi\>\xi
\ee
for all $X\in \mathfrak{X}(M)$,
\be\label{eq:pi2}
((1-t)I-R)\xi=0 \,\,\,\mbox{and}\,\,\,\,t(1-t)=\|\xi\|^2.
\end{equation}
It is convenient to introduce the one-form $S=S_f=\xi^\flat\in \Gamma(T^*M)$, that is, $S(X)=\<\xi, X\>$ for all $X\in \mathfrak{X}(M)$.  The definition of
$R$ in \eqref{L} implies that it is a symmetric endomorphism.  
It follows from \eqref{eq:pi1} that $\ker S$ can be pointwise  orthogonally decomposed as 
\be\label{kerS}
\ker S=\ker R\oplus \ker(I-R).
\ee 
 In Lemma 3.2 of \cite{MT} it was shown that these subspaces give rise to smooth subbundles of $TM$ in open subsets where their dimensions are constant.
Since $\dim \ker S(x)\geq n-1$ at any $x\in M^n$, either $S(x)=0$, in which case $0$ and $1$ are the only eigenvalues of $R(x)$, or there exists a unique further eigenvalue of $R(x)$, which by \eqref{eq:pi2} is equal to $r=r_f=1-t\in (0,1)$ and has $\mbox{span}\, \{\xi\}$ as its eigenspace. Notice that
    \begin{equation}\label{dimkerR}
    k-1\leq \dim \ker R\leq k \,\,\,\,\mbox{and}\,\,\,\,n-k\leq \dim \ker (I-R)\leq n-k+ 1.
    \end{equation}
Indeed, since $\ker R\subset \ker S$, it follows from \eqref{L} that  $X\in \ker R$ if and only if $\pi_{2}\circ f_{*}X=0$. 
Thus $f_*\ker R(x)= f_*T_xM\cap T_{\pi_{1}\circ f(x)}\mathbb{Q}_{c_1}^{k}$, where $\pi_1$ is the projection onto $\mathbb{Q}_{c_1}^{k}$, and since $f_*T_xM$ has codimension $1$ in $T_{f(x)}(\mathbb{Q}_{c_1}^{k}\times \mathbb{Q}_{c_2}^{n-k+1})$,  the inequalities $k-1\leq \dim \ker R\leq k$ follow. Similarly, one argues that $n-k\leq \dim \ker (I-R)\leq n-k+ 1$.

 We say that a hypersurface $f\colon M^{n}\rightarrow \mathbb{Q}_{c_1}^{k}\times \mathbb{Q}_{c_2}^{n-k+1}$ \emph{splits}  if either $M^n$ is (isometric to)  an open subset of $N^{k-1}\times \mathbb{Q}_{c_2}^{n-k+1}$ and $f=f_{1}\times i_{2}:N^{k-1}\times \mathbb{Q}_{c_2}^{n-k+1}\to \mathbb{Q}_{c_1}^{k}\times \mathbb{Q}_{c_2}^{n-k+1}$, or $M^n$ is (isometric to) an open subset of $\mathbb{Q}_{c_1}^{k}\times N^{n-k}$ and $f=i_{1}\times f_{2}\colon \mathbb{Q}_{c_1}^{k}\times N^{n-k}\to  \mathbb{Q}_{c_1}^{k}\times \mathbb{Q}_{c_2}^{n-k+1}$, where $f_{1}\colon N^{k-1}\rightarrow \mathbb{Q}_{c_1}^{k}$ and $f_{2}\colon N^{n-k}\rightarrow  \mathbb{Q}_{c_2}^{n-k+1}$ are hypersurfaces and $i_{1}$ and $i_{2}$ are the respective identity maps. 

The following fact is a consequence of Proposition $3.3$ in \cite{MT}.

\begin{proposition}\label{splits}
     For a hypersurface $f\colon M^{n}\rightarrow \mathbb{Q}_{c_1}^{k}\times \mathbb{Q}_{c_2}^{n-k+1}$, $2\leq k\leq n-1$, the following are equivalent:
     \begin{itemize}
         \item[(i)] $f$ splits locally;
         \item[(ii)] $S_{f}=0$;
         \item[(iii)] $\dim\ker R = k$ or $\dim\ker (I-R)=n-k+1$ everywhere.
     \end{itemize}
\end{proposition}

Now we list some differential equations relating $R$, $\xi$ and $t$ that follow by computing the tangent and normal components in the equation 
$\nab\pi_2=\pi_2\nab$, where $\nab$ is the Levi-Civita connection of the ambient space,
applied to both tangent and normal vectors, and using the Gauss and Weingarten formulae:
\be\label{derR}
(\nabla_XR)Y= \<Y, \xi\>AX+\<AX,Y\>\xi,
\ee
\be\label{derS2}
\nabla_X\xi=(tI-R)AX
\ee
for all $X, Y\in\mathfrak{X}(M)$, and 
\be\label{dert}
\grad t=-2A\xi.
\ee
In terms of $S$, Eq. \eqref{derS2} takes the form
\be\label{derS}
(\nabla_XS)Y:=X(S(Y))-S(\nabla_XY)=\<(tI-R)AX,Y\>.
\ee

The Gauss and Codazzi equations of $f$ are
\be\label{GaussA}
\mathcal{R}(X,Y)=c_1(X\wedge Y-X\wedge RY-RX\wedge Y) + (c_1+ c_2)RX\wedge RY +AX\wedge AY,
\ee
and 
\be\label{codazziA}
(\nabla_XA)Y-(\nabla_YA)X=
(c_1(I-R)-c_2R)(X\wedge Y)\xi
\ee
for all $X,Y\in\mathfrak{X}(M)$, where $\mathcal{R}$ is the curvature tensor of $M^n$.\vspace{1ex}

Theorem \ref{exist} below is the  version for hypersurfaces of the existence part of Theorem 2.2 in \cite{LTV}.

\begin{theorem} \label{exist} Let $M^n$ be a simply connected Riemannian manifold, let $A$ and  $R$ be symmetric endomorphisms, the latter having $0$, $r\in (0,1)$ and $1$ as eigenvalues,  with $r$ simple, and let $S\in \Gamma(T^*M)$  be a one-form  such that \eqref{derR}, \eqref{derS2}, \eqref{dert}, \eqref{GaussA} and \eqref{codazziA} are satisfied, with $\xi=S^{\#}$ and $t=1-r$.
Then there exists an isometric immersion  $f\colon M^n \to \mathbb{Q}^k_{c_1}\times \mathbb{Q}^{n-k+1}_{c_2}$ and a unit normal vector field $N$ along $f$ such that $A=A_N^f$, $R=R^f$ and $S=S^f$.
\end{theorem}

Two hypersurfaces $f,g\colon M^n\to\mathbb{Q}^k_{c_1}\times \mathbb{Q}^{n-k+1}_{c_2}$ are \emph{congruent} if there exists an isometry $\Phi$ of $ \mathbb{Q}^k_{c_1}\times \mathbb{Q}^{n-k+1}_{c_2}$ such that $g=\Phi\circ f$. The uniqueness part of 
Theorem 2.2 in \cite{LTV} for hypersurfaces can be stated as follows.

\begin{proposition} \label{cong} 
The oriented hypersurfaces $f,g\colon M^n\to\mathbb{Q}^k_{c_1}\times \mathbb{Q}^{n-k+1}_{c_2}$ are congruent if and only if either
    $A_g=A_f$, $R_g=R_f$ and $S_g=S_f$, or $A_g=-A_f$, $R_g=R_f$ and $S_g=-S_f$.
    \end{proposition}

Our interest in this paper is on isometric immersions $f\colon M^n\to \Hy^{k}\times\Sf^{n-k+1}$, with $2\leq k\leq n-1$, in which case the Gauss and Codazzi equations become, respectively,  
\be\label{Gauss}
R(X,Y)=X\wedge RY+RX\wedge Y-X\wedge Y+AX\wedge AY
\ee
and 
\be\label{codazzi}
(\nabla_XA)Y-(\nabla_YA)X+(X\wedge Y)\xi=0
\ee
for all $X,Y\in\mathfrak{X}(M)$.

Now consider the standard embedding $i\colon \Hy^k\times\Sf^{n-k+1}\to \Les^{k+1}\times\R^{n-k+2}=\Les^{n+3}$ and denote $\tilde{f}=i\circ f$.
Decompose the position vector $\tilde{f}$ as $\tilde{f}=f_1+f_2$, where $f_1=\pi_1\circ f$ and $f_2=\pi_2\circ f$ are the components in $\Les^{k+1}$ and $\R^{n-k+2}$, respectively.
Notice that $\<\tilde{f},\tilde{f}\>=0$, for $\<f_1,f_1\>=-1$ and $\<f_2,f_2\>=1$.
The second fundamental form $\a_{\tilde{f}}$ of $\tilde{f}$ is given by
\begin{align}
    \a_{\tilde{f}}(X,Y)&=\<AX,Y\>i_*N+(\<X,Y\>-\<RX,Y\>)f_1-\<RX,Y\>f_2\nonumber\\
    &=\<AX,Y\>i_*N+\<X,Y\>f_1-\<RX,Y\>\tilde{f}.\label{sfft}
\end{align}

\section{Rotational and doubly rotational hypersurfaces}\label{Sec:Rotational}

In this section, we define rotational and doubly rotational hypersurfaces in $\mathbb{H}^k\times \mathbb{S}^{n-k}$, $2\leq k\leq n-1$, and derive a N\"olker-type theorem characterizing them as isometric immersions of warped products satisfying some conditions. 

Let $h_1\colon M_1^r\to\Hy^k\subset\Les^{k+1}$ and $h_2\colon M_2^s\to\Sf^{l}\subset\R^{l+1}$ be umbilical immersions (extrinsic circles if $r=1$ or $s=1$). It is well known that $h_1(M_1)$ and $h_2(M_2)$ are open subsets of  the intersections $\Hy^k\cap L_1$ and $\Sf^{l}\cap L_2$ of $\Hy^k$ and $\Sf^{l}$ with affine subspaces $L_1$ and $L_2$ of $\Les^{k+1}$ and $\R^{l+1}$, respectively. According to whether  the linear subspace parallel to $L_1$ is  space-like, time-like, or degenerate, we refer to these cases as \emph{elliptic}, \emph{hyperbolic} and \emph{parabolic}, respectively.\vspace{1ex}

\textbf{Elliptic case:} In this case, $h_1(M_1)$ is an open subset of a sphere of constant sectional curvature $\kappa_1>0$ in $\Hy^k$ (of a circle of radius $\kappa_1^{-1/2}$ if $r=1$). Let $\tilde{h}_1$ denote the composition of $h_1$ with the inclusion of $\Hy^k$ in $\Les^{k+1}$, and let $L_1^0$ be the linear subspace parallel to $L_1$. 
Then, the position vector $\tilde{h}_1$ can be orthogonally decomposed as $\tilde{h}_1=c_1+h_1^0$, where $h_1^0\in L_1^0$, $\<h_1^0,h_1^0\>=1/\kappa_1$ and $c_1\in L_1\subset\Les^{k+1}$ is the center of this sphere (circle). 
Notice  that $c_1$ is a time-like vector orthogonal to $L_1^0$, with $\<c_1,c_1\>=-(1+1/\kappa_1$). 

The immersion $\tilde{h}_1$ has flat normal bundle. We take a parallel orthonormal frame $\{\eta_0,\eta_1,\dots,...,\eta_{k-r}\}$ of $N_{\tilde{h}_1}M_1$, where $\eta_0=\sqrt{\kappa_1}h_1^0$ and $\eta_1=(1/\sqrt{-\<c_1,c_1\>})c_1$. 
Notice that the vector fields $\eta_i$, $1\leq i\leq k-r$, are constant. 
Arguing similarly for $h_2$, we take a parallel orthonormal frame $\{\delta_0,\delta_1,\dots,\delta_{l-s}\}$ of $N_{\tilde{h}_2}M_2$, where $\delta_0=\sqrt{\kappa_2}\,h_2^0$ and $\delta_1=c_2/\sqrt{\<c_2,c_2\>}$.

Let $h\colon M_3:=M_1\times M_2\to \Hy^k\times\Sf^l$ be the extrinsic product of $h_1$ and $h_2$, given by $h(x,y)=(h_1(x),h_2(y))$, and let $\tilde{h}$ denote its composition with the inclusion of $\Hy^k\times\Sf^l$ in $\Les^{k+l+2}=\Les^{k+1}\oplus\R^{l+1}$. 
Then $\tilde{h}$ has flat normal bundle, and we have a parallel orthonormal frame $\{\eta_0,\dots,\eta_{k-r},\delta_0,\dots,\delta_{l-s}\}$ of $N_{\tilde{h}}M_3$, with $\eta_i$ and $\delta_j$ constant for $i\neq 0\neq j$.
Fix orthonormal bases $\{e_i\}_{0\leq i\leq k-r}$ and $\{d_j\}_{0\leq j\leq l-s}$ of $\Les^{k-r+1}$ and $\R^{l-s+1}$ respectively, where $e_1$ is time-like. Then we have a parallel vector bundle isometry 
$$
\phi\colon M_3\times(\Les^{k-r+1}\oplus\R^{l-s+1})\to N_{\tilde{h}}M_3
$$
such that $\phi((x_1,x_2),e_i)=\phi_{(x_1,x_2)}(e_i)=\eta_i$ and $\phi_{(x_1,x_2)}(d_j)=\delta_j$, for all $i,j$.
\vspace{2ex}

\textbf{Hyperbolic case:} In this case, $h_1(M_1)$ is an open subset of a hyperbolic space of constant sectional curvature $-1<\kappa_1<0$ in $\Hy^k\subset \mathbb{L}^{k+1}$ (of an extrinsic circle of radius $(-\kappa_1)^{-1/2}$ if $r=1$). 
Let $L_1^0$ be the linear subspace of $\mathbb{L}^{k+1}$ parallel to $L_1$, and let $c_1$ be a space-like vector, orthogonal to $L_1^0$ such that $L_1=c_1+L_1^0$. 
Then, the position vector $\tilde{h}_1$ can be orthogonally decomposed as $\tilde{h}_1=c_1+h_1^0$, with $h_1^0\in L_1^0$ and $\<h_1^0,h_1^0\>=1/\kappa_1$.
Arguing as before, we take a parallel orthonormal frame $\{\eta_0,\eta_1,\dots,...,\eta_{k-r}\}$ of $N_{\tilde{h}_1}M_1$, where $\eta_0=\sqrt{-\kappa_1}h_1^0$ and $\eta_1=c_1/\sqrt{\<c_1,c_1\>}$. Notice that $\eta_i$, $1\leq i\leq k-r$ are constant and that $\eta_0$ is time-like.

Let $h$ be the extrinsic product of $h_1$ and $h_2$, and let $\tilde{h}$ denote its composition with the inclusion of $\mathbb{H}^k\times \mathbb{S}^{l}$ into $\Les^{k+l+2}=\Les^{k+1}\oplus\R^{l+1}$.
Fix orthonormal bases $\{e_i\}_{0\leq i\leq k-r}$ and $\{d_j\}_{0\leq j\leq l-s}$ of $\Les^{k-r+1}$ and $\R^{l-s+1}$ respectively, where $e_0$ is time-like. 
Then we have a parallel vector bundle isometry 
$$
\phi\colon M_3\times(\Les^{k-r+1}\oplus\R^{l-s+1})\to N_{\tilde{h}}M_3
$$
such that $\phi((x_1,x_2),e_i)=\phi_{(x_1,x_2)}(e_i)=\eta_i$ and $\phi_{(x_1,x_2)}(d_j)=\delta_j$, for all $i,j$.
\vspace{2ex}

\textbf{Parabolic case:} The linear subspace $L_1^0$ parallel to $L_1$ is degenerate, hence it contains a constant null direction. Let $0\neq v_0\in L_1^0$ be a null vector such that $\<\tilde{h}_1(x),v_0\>=1$ for all $x\in M_1$, and notice that $\tilde{h}_1^0=\tilde{h}_1+\frac{1}{2}v_0$ satisfies
$$
\<\tilde{h}_1^0,\tilde{h}_1^0\>=0 \;\;\mbox{and}\;\; \<\tilde{h}_1^0,v_0\>=1.
$$
Moreover, $\tilde{h}_1^0$ lies in the normal bundle of $\tilde{h}_1$. Take a parallel frame of $N_{\tilde{h}_1}M_1$ given by $\{\eta_0,\eta_1,\eta_2,\dots,\eta_{k-r}\}$, where $\eta_0=\tilde{h}_1^0$, $\eta_1=v_0$ and $\eta_i$, $2\leq i\leq k-r$, are constant.
Let $h$ and $\tilde{h}$ be given as before.
Take a pseudo-orthonormal basis $\{e_0,e_1,\dots,e_{k-r}\}$ of $\Les^{k-r+1}$, that is, the vectors $e_i$ satisfy
$\<e_i,e_i\>=0$ for $i=0,1$, $\<e_0,e_1\>=1$,  $\<e_i,e_j\>=\delta_{ij}$ for $i,j\notin\{0,1\}$ and $\<e_i, e_k\>=0$ if $k\in \{0,1\}$ and $i\not\in \{0, 1\}$.
Then we have a parallel vector bundle isometry
$$
\phi\colon M_3\times(\Les^{k-r+1}\oplus\R^{l-s+1})\to N_{\tilde{h}}M_3
$$
such that $\phi((x_1,x_2),e_i)=\phi_{(x_1,x_2)}(e_i)=\eta_i$ and $\phi_{(x_1,x_2)}(d_j)=\delta_j$, for all $i,j$.
\vspace{2ex}

In any case, the isometry $\phi$ can be explicitly expressed as
\be\label{phi}
\phi_{(x,y)}(\sum_i a_ie_i+\sum_j b_jd_j)=\sum_i a_i\eta_i+\sum_jb_j\delta_j.
\ee
Moreover, notice that $\sum_i a_i\eta_i\in\Les^{k+1}$ and $\sum_jb_j\delta_j\in\R^{l+1}$.

Regarding $\phi$ as a map from $ (\Les^{k-r+1}\oplus\R^{l-s+1})\times M_1\times M_2$ into $\Les^{k+1}\oplus\R^{s+1}$, its restriction to $ (\Hy^{k-r}\times\Sf^{l-s})\times M_1\times M_2$ induces a warped product metric on the subset $V$ of regular points, hence  determines a warped product representation $\Psi\colon V\subset  (\Hy^{k-r}\times\Sf^{l-s})\times_{\sigma_1} M_1\times_{\sigma_2} M_2\to \Psi(V)\subset \Hy^k\times\Sf^l$.
Our goal now is to use $\phi$ to build hypersurfaces that are ``adapted" to such warped product representation.

Let $f_0\colon M_0\to\Hy^{k-r}\times\Sf^{l-s}$ be a hypersurface, and let $\tilde{f_0}$ denote its composition with the inclusion of $\Hy^{k-r}\times\Sf^{l-s}$ into $\Les^{k-r+1}\oplus\R^{l-s+1}$. 
Define the map
$\tilde{f}\colon M_0\times M_1\times M_2\to\Les^{k+l+2}$ by
\be\label{tildef}
\tilde{f}(x_0,x_1,x_2)=\phi_{(x_1,x_2)}(\tilde{f}_0(x_0)).
\ee
Since $f_0$ is an immersion into $\Hy^{k-r}\times\Sf^{l-s}$, then $\tilde{f}(x_0,x_1,x_2)\in\Hy^k\times\Sf^l$ by \eqref{tildef}. 
Thus $\tilde f$  determines a map $f\colon M_0\times M_1\times M_2\to\Hy^k\times\Sf^l$.
Fix orthonormal bases $\{d_j\}_{0\leq j\leq l-s}$ of $\R^{l-s+1}$ and $\{e_i\}_{0\leq i\leq k-r}$ of $\Les^{k-r+1}$, with $e_1$ time-like in the elliptic case and $e_0$ time-like in the hyperbolic case. In the parabolic case, we fix a pseudo-orthonormal basis $\{e_i\}_{0\leq i\leq k-r}$.
With respect to these bases, we write 
$$
\tilde{f}_0(x_0)=\sum_ia_i(x_0)e_i+\sum_jb_j(x_0)d_j.
$$

Given $X_i\in\mathfrak{X}(M_i)$, $0\leq 1\leq 2$, to keep the notation simple,
we also denote by $X_i$ the vector field of $M_0\times M_1\times M_2$ that is projected to $X_i$ in the $i^{th}$ factor and to $0$ in the other ones.
Let $\hat{\nabla}$ denote the Levi-Civita connection of $\Les^{k+l+2}$, and let $\phi_{(x_1,x_2)}$ be denoted simply by $\phi$ whenever it is not necessary to explicitly specify the pair $(x_1,x_2)$.
By \eqref{phi}, we have
$$
\tilde{f}_*X_0=\phi(\tilde{f}_{0*}X_0)
$$
and
\begin{align*}
\tilde{f}_*X_2&=\hat{\nabla}_{X_2}\phi(\sum_ia_i(x_0)e_i+\sum_jb_j(x_0)d_j)\\
&=\hat{\nabla}_{X_2}(\sum_ia_i(x_0)\eta_i+\sum_jb_j(x_0)\delta_j)\\
&=\hat{\nabla}_{X_2}(\sum_jb_j(x_0)\delta_j)).
\end{align*}

From the definition of the frames $\{\delta_j\}$, we see that all $\delta_j$'s but $\delta_0$ are constant. Thus
$$
\tilde{f}_*X_2=b_0\hat{\nabla}_{X_2}\delta_0.
$$
The vector field $\delta_0$ is given by $\delta_0(x_2)=\sqrt{\kappa_2}(h_2(x_2)-c_2)$, where $c_2$ is constant, hence
$$
\tilde{f}_*X_2=\sqrt{\kappa_2}\<\tilde{f}_0,d_0\>h_{2*}X_2.
$$
Using the same arguments, we obtain
$$
\tilde{f}_*X_1=\begin{cases}
\begin{array}{ll}
  \sqrt{\kappa_1}\<\tilde{f}_0,e_0\>h_{1*}X_1,   & \mbox{elliptic case} \\
    -\sqrt{-\kappa_1}\<\tilde{f}_0,e_0\>h_{1*}X_1, & \mbox{hyperbolic case}\\
    \<\tilde{f}_0,e_1\>h_{1*}X_1,  &   \mbox{parabolic case}.
\end{array}
\end{cases}
$$
In particular, in the open subset where $f$ is an immersion, it determines an isometric immersion of a warped product 
$$
f\colon M_0\times_{\rho_1}M_1\times_{\rho_2}M_2\to\Hy^k\times\Sf^{l}.
$$
The corresponding warping functions
$\rho_1,\rho_2\in C^\infty(M_0)$ are given by 
\be\label{rho1}
\rho_1=\begin{cases}
\begin{array}{ll}
  \sqrt{\kappa_1}\<\tilde{f}_0,e_0\>,   & \mbox{elliptic case} \\
    \sqrt{-\kappa_1}\<\tilde{f}_0,e_0\>, & \mbox{hyperbolic case}\\
    \<\tilde{f}_0,e_1\>,  &   \mbox{parabolic case}
\end{array}
\end{cases}
\ee
and 
\be\label{rho2}
\rho_2=\sqrt{\kappa_2}\<\tilde{f}_0,d_0\>.
\ee

From now on, we assume that $f$ is an immersion.
Let $N_{f_0}$ be a (local) unit vector field normal to $f_0$ in $\Hy^{k-r}\times\Sf^{l-s}$. 
Then $N_f=\phi(N_{f_0})$ is normal to the immersion $f$ into $\Hy^k\times\Sf^{l}$.
If $N_{f_0}$ is written, with respect to our fixed frames $\{e_i\}_{1\leq i\leq k-r}$ and $\{d_j\}_{1\leq j\leq l-s}$, as
$$
N_{f_0}=\sum_i\mu_ie_i+\sum_j\nu_jd_j,
$$
then $N_f$ has the expression
\be\label{Nf coord}
N_f=\sum_i\mu_i\eta_i+\sum_j\nu_j\delta_j.
\ee
Notice that the functions $\mu_i$ and $\nu_j$ depend only on $x_0$ and that they are, up to sign, inner products of $N_{f_0}$ with some element of the fixed basis.
It follows from the definition of $N_f$ that
$$
\hat{\nabla}_{X_0}N_f=\phi(\bar{\nabla}_{X_0}N_{f_0}),
$$
where $\bar{\nabla}$ is the Levi-Civita connection in $\Les^{k-r+l-s+2}$.
Let $A_f$ denote the shape operator of $f$ with respect to $N_f$ and $A_{f_0}$ that of $f_0$ with respect to $N_{f_0}$.
It follows from the preceding equation that 
\be\label{A0}
\tilde{f}_*A_fX_0=\phi(\tilde{f}_{0*}A_{f_0}X_0).
\ee
Taking into account that, except for $\eta_0$, all $\eta_i$'s are constant, the derivative of $N_f$ in the direction of $X_1$ is given by
$$
\hat{\nabla}_{X_1}N_f=\begin{cases}
    \begin{array}{ll}
         \sqrt{\kappa_1}\<N_{f_0},e_0\>h_{1*}X_1, & \mbox{elliptic case}\\
         -\sqrt{-\kappa_1}\<N_{f_0},e_0\>h_{1*}X_1,&   \mbox{hyperbolic case}\\
         \<N_{f_0},e_1\>h_{1*}X_1, & \mbox{parabolic case}.
    \end{array}
\end{cases}
$$
 Therefore $X_1$ is a principal direction of $f$ with principal curvature
\be\label{princurv1}
 \lambda_1=\begin{cases}
     \begin{array}{ll}
         -\frac{\<N_{f_0},e_0\>}{\<\tilde{f}_0,e_0\>}, & \mbox{elliptic and hyperbolic cases}  \\
         -\frac{\<N_{f_0},e_1\>}{\<\tilde{f}_0,e_1\>}, & \mbox{parabollic case}.
     \end{array}
 \end{cases}
 \ee
Using the same arguments, we see that $X_2$ is a principal direction of $f$ with principal curvature
\be\label{princurv2}
\lambda_2=-\frac{\<N_{f_0},d_0\>}{\<\tilde{f}_0,d_0\>}.
\ee
Let $\{R_f, S_f, \xi_f,t_f\}$  and $\{R_{f_0}, S_{f_0}, \xi_{f_0}, t_{f_0}\}$ be associated with $f$ and $f_0$, respectively. 
Let $\pi_2$ denote both the projections onto $\Sf^{l}$ and $\Sf^{l-s}$. It follows from \eqref{phi} that
$$
\pi_2\phi(f_{0*}X_0)=\phi(\pi_2f_{0*}X_0).
$$
Thus 
$$
R_fX_0=R_{f_0}X_0
$$
for any $X_0\in\mathfrak{X}(M_0)$.
Since $N_f=\phi(N_{f_0})$, then
$$
f_*\xi_f+t_fN_f=\pi_2N_f=\pi_2\phi(N_{f_0})=\phi(f_{0*}\xi_{f_0}+t_{f_0}N_{f_0}).
$$
Hence
$\xi_f=\phi(\xi_{f_0})$ and $t_f=t_{f_0}$. In particular, $S_f(X_0)=S_{f_0}(X_0)$ for any $X_0\in\mathfrak{X}(M_0)$.
\vspace{2ex}

If we start with only one umbilical immersion, either into $\Hy^k$ or $\Sf^l$, we have a similar construction. 
In this case, we obtain an isometric immersion $f\colon M_0\times_\rho N\to \Hy^k\times\Sf^l$, where 
the warping function $\rho$ and the shape operator $A_f$ are given by the same expressions above, according to whether the umbilical immersion lies in $\Hy^k$ or $\Sf^l$.

\begin{remark}
\noindent{\em 
The hypersurfaces given by the preceding construction (in case $h$ is an extrinsic product $h=h_1\times h_2$) are  called \textit{doubly rotational} because they can be generated by the action of a certain subgroup of the group of isometries of $\mathbb{H}^k\times \mathbb{S}^l$ on the ``profile" $\phi_{(x_1^0, x_2^0)}(\tilde f_0(M_0))$ for fixed $x_1^0\in M_1$ and $x_2^0\in M_2$, the  orbit 
(the ``parallel") through $\phi_{(x_1^0, x_2^0)}(\tilde f_0(x_0^0))$, $x_0^0\in M_0$,  being the extrinsic product $(x_1, x_2)\mapsto \phi_{(x_1, x_2)}(\tilde f_0(x_0^0))$.}
\end{remark}

Our next goal is to describe sufficient conditions on an isometric immersion of a warped product $f\colon M_0\times_{\rho_1}M_1\times_{\rho_2}M_2\to\Hy^{k}\times\Sf^{l}$ which ensure that the immersion is given by \eqref{tildef}.  
The inclusions of each factor $M_i$ in $M=M_0\times_{\rho_1}M_1\times_{\rho_2}M_2$ determine the tangent distributions $\Delta_i\subset TM$, $i=0,1,2$. 
These integrable distributions give rise to the product net $TM=\Delta_0\oplus \Delta_1\oplus \Delta_2$. 
While $\Delta_1$ and $\Delta_2$ are spherical, the distribution $\Delta_0$ is totally geodesic. 
Moreover, $\Delta_0\oplus \Delta_1$ and $\Delta_0\oplus \Delta_2$ are also totally geodesic.

\begin{proposition}\label{propmulti}
    Let $f\colon M^n\to\Hy^k\times\Sf^l$ be an isometric immersion of a warped product $M= M_0\times_{\rho_1}M_1\times_{\rho_2}M_2$, where
    $\dim M_1=r\geq 0$, $\dim M_2=s\geq0$, $(r,s)\neq (0,0)$ and $\dim M_0=k+l-r-s-1$. Let $N_f$ be a unit vector field normal to $f$ with corresponding shape operator $A_f$. Assume that $\Delta_1$ and $\Delta_2$ are eigenspaces of $A_f$ such that $\Delta_1\subset \ker R_f$ and that $\Delta_2\subset \ker (I-R_f)$.
    Then, there exist umbilical immersions $h_1\colon M_1\to\Hy^k$ and $h_2\colon M_2\to\Sf^{l}$ if $rs\neq0$, or just an umbilical immersion into $\Hy^k$ or $\Sf^l$ if $rs=0$,  and an isometric immersion $f_0\colon M_0\to \Hy^{k-r}\times\Sf^{l-s}$, such that $f$ (regarded as an isometric immersion into  $\Les^{k+l+2}$) is given by \eqref{tildef}. 
\end{proposition}
\proof We argue for the case $rs\neq 0$; the arguments in the other cases are similar.
Fix $\bar x_0\in M_0$. Assume, without loss of generality, that $\rho_1(\bar{x}_0)=1=\rho_2(\bar{x}_0)$ (this can always be achieved by scaling the metrics of $M_1$ and $M_2$, if necessary). Then the map $h=h_{\bar{x}_0}\colon M_1\times M_2\to\Hy^k\times\Sf^l$, given by
$$
h(x_1,x_2)=f(\bar{x}_0,x_1,x_2) 
$$ 
for all $(x_1, x_2)\in M_1\times M_2$, is an isometric immersion with respect to the Riemannian product metric on $M_1\times M_2$. 

Let $(\bar{\Delta}_1, \bar{\Delta}_2)$ denote the product net of $M_1\times M_2$. Then ${\nu_{\bar{x}_0}}_*\bar{\Delta}_i(x_1, x_2)=\Delta_i(\bar{x}_0, x_1, x_2)$ for all $(x_1, x_2)\in M_1\times M_2$, $1\leq i\leq 2$, where $\nu_{\bar{x}_0}$ denotes the inclusion 
$(x_1, x_2)\in M_1\times M_2\mapsto (\bar{x}_0, x_1, x_2)\in M_0\times M_1\times M_2.$
Since $\Delta_1\subset \ker R_f$ and $\Delta_2\subset \ker (I-R_f)$, then 
$$
{h}_*X_1\in T_{\tilde{\pi}_1({h}(x_1, x_2))}\mathbb{H}^k\subset  \mathbb{L}^{k+1}\quad \mbox{and}\quad {h}_*X_2\in T_{\tilde{\pi}_2({h}(x_1, x_2))}\mathbb{S}^{l}\subset  \mathbb{R}^{l+1}
$$  
for all $X_1\in \bar{\Delta}_1(x_1, x_2)$ and $X_2\in \bar{\Delta}_2(x_1, x_2)$, where $\tilde{\pi}_1\colon \mathbb{H}^k \times \mathbb{S}^l\to \mathbb{H}^k$ and  $\tilde{\pi}_2\colon \mathbb{H}^k \times \mathbb{S}^l\to  \mathbb{S}^l$ are the projections. 
Therefore, denoting by $\tilde{h}$ the composition of $h$ with the inclusion $i\colon \mathbb{H}^k\times \mathbb{S}^l \to\Les^{k+l+2}$, we have
$$
P_1\tilde{h}_*X_2=0=P_2 \tilde{h}_*X_1
$$  
for all $X_1\in \bar{\Delta}_1$ and $X_2\in \bar{\Delta}_2$, where $P_1\colon \mathbb{L}^{k+1}\times \mathbb{R}^{l+1}\to  \mathbb{L}^{k+1}$ and $P_2\colon \mathbb{L}^{k+1}\times \mathbb{R}^{l+1}\to  \mathbb{R}^{l+1}$  are orthogonal projections.
Therefore,  the map  $P_i\circ \tilde{h}$  is 
constant along the fibers of the projection $\pi_i\colon M_1\times M_2\to M_i$, $1\leq i\leq 2$. For a fixed $(\bar{x}_1, \bar{x}_2) \in M_1\times M_2$, consider the maps $\tilde{h}_1\colon M_1\to \mathbb{L}^{k+1}$ and $\tilde{h}_2\colon M_2\to \R^{l+1}$ given by 
$$\tilde{h}_1(x_1)=P_1(\tilde{h}(x_1, \bar{x}_2))\quad \mbox{and} \quad 
\tilde{h}_2(x_2)=P_2(\tilde{h}(\bar{x}_1, x_2)).$$ Then
$$
P_i\circ \tilde{h} =\tilde{h}_i\circ \pi_i,\;\;1\leq i\leq 2,
$$ 
that is,
$$\tilde{h}(x_1,x_2)=(\tilde{h}_1(x_1), \tilde{h}_2(x_2)).$$
Since $\tilde{h}$ takes values in $\Hy^k\times\Sf^l\subset \mathbb{L}^{k+l+2}$, then $\tilde{h}_1(M_1)\subset \Hy^k$ and  
$\tilde{h}_2(M_2)\subset \Sf^l$, hence  $\tilde{h}_1$ and $\tilde{h}_2$ give rise to isometric immersions 
$h_1\colon M_1\to \Hy^k$ and $h_2\colon M_2\to \Sf^l$ such that
$$h(x_1, x_2)=(h_1(x_1), h_2(x_2)).$$
Our assumption that $\Delta_1$ and $\Delta_2$ are eigenspaces of $f$ and the fact that they are spherical distributions imply that $h_1$ and $h_2$ are umbilical immersions (or extrinsic circles).

Now, the assumptions on $A_f$ and $R_f$, together with \eqref{sfft}, imply that $\a_{\tilde{f}}$ is adapted to the product net of $M$. 
Since, in addition,   $\Delta_{0}$ is totally geodesic, then $\tilde{f}_*(\Delta_1\oplus\Delta_2)$ is constant in $\mathbb{L}^{k+l+2}$ along each leaf of $\Delta_0$ (see Lemma 10.13 in \cite{DT}).
Fix $(\bar{x}_1,\bar{x}_2)\in M_1\times M_2$ and consider the immersion $f(\cdot,\bar{x}_1,\bar{x}_2)$ of the leaf of $\Delta_0$ through $(\bar{x}_1,\bar{x}_2)$. 
Then the image of $M_0$ by $f(\cdot,\bar{x}_1,\bar{x}_2)$ lies in the affine normal space of $\tilde{h}$ at $(\bar{x_1},\bar{x}_2)$, that is,
$$
\tilde{f}(x_0,\bar{x}_1,\bar{x}_2)\in \tilde{h}(\bar{x}_1,\bar{x}_2)+ N_{\tilde{h}}(M_1\times M_2)(\bar{x}_1,\bar{x}_2)
$$
for all $x_0\in M_0$ .
Hence, for each $x_0\in M_0$,  we can regard
$$
\xi^{x_0}(x_1,x_2)=\tilde{f}(x_0,x_1,x_2)-\tilde{h}(x_1,x_2)
$$
as a section of $N_{\tilde{h}}(M_1\times M_2)$.
For any $Z\in\mathfrak{X}(M_1\times M_2)$ we have
\begin{align*}
\nab_Z\xi^{x_0}&=\tilde{f}_*(x_0,x_1,x_2)Z- \tilde{h}_*(x_1,x_2)Z \\
&=\tilde{f}_*(x_0,x_1,x_2)Z- \tilde{f}_*(\bar{x}_0,x_1,x_2)Z.
\end{align*}
Using that $\tilde{f}_*(\Delta_1\oplus\Delta_2)$ is constant along each leaf of $\Delta_0$, by the previous expression we have 
$$
\nab_Z\xi^{x_0}\in \tilde{h}_*T_{(x_1,x_2)}(M_1\times M_2).
$$
Hence $\xi^{x_0}$ is a parallel section of $N_{\tilde{h}}(M_1\times M_2)$.
Observe that $\tilde{h}$ is an extrinsic product of umbilical immersions and that its  normal bundle
$$
N_{\tilde{h}}(M_1\times M_2)=N_{\tilde{h}_1}M_1\oplus N_{\tilde{h}_2}M_2
$$
is flat. Moreover, there is an isometry 
$$
\phi\colon(M_1\times M_2)\times (\Les^{k+1-r}\oplus\R^{l-s+1})\to N_{\tilde{h}_1}M_1\oplus N_{\tilde{h}_2}M_2
$$
satisfying
\be\label{phipreser}
\phi_{(x_1,x_2)}(e,0)=(\eta,0) \;\;\mbox{and}\;\; \phi_{(x_1,x_2)}(0,d)=(0,\delta)
\ee
according to the orthogonal decompositions $\Les^{k-r+1}\oplus\R^{l-s+1}$ and $N_{\tilde{h}_1}M_1\oplus N_{\tilde{h}_2}M_2$.
Since $\xi^{x_0}$ is parallel, then there is a well-defined map $\bar{f}_0\colon M_0\to\Les^{k-r+1}\oplus\R^{l-s+1}$ such that
$$
\phi_{(x_1,x_2)}(\bar{f}_0(x_0))=\xi^{x_0}(x_1,x_2)
$$
for all $(x_1,x_2)\in M_1\times M_2$.
We have 
$$
\tilde{f}_*(x_0,x_1,x_2)X_0=\phi_{(x_1,x_2)}(\bar{f}_{0*}(x_0)X_0)
$$
for any $X_0\in\mathfrak{X}(M_0)$. Thus $\bar{f}_0$ is an isometric immersion.
It follows from the definition of $\xi^{x_0}$ that
$$
\tilde{f}(x_0,x_1,x_2)=\tilde{h}(x_1,x_2)+\phi_{(x_1,x_2)}(\bar{f}_0(x_0)).
$$
Notice that $\tilde{h}(x_1,x_2)$ can be seen as a parallel section of $N_{\tilde{h}}(M_1\times M_2)$,
hence there is a vector $v_0\in\Les^{k-r+1}\oplus\R^{l-s+1}$ such that $\phi(v_0)=\tilde{h}$.
Moreover, we can regard $N_{\tilde{h}_1}M_1(x_1)$ and $N_{\tilde{h}_2}M_2(x_2)$ as affine subspaces of $\Les^{k+1}$ and $\R^{l+1}$ based at $\tilde{h}_1(x_1)$ 
and $\tilde{h}_2(x_2)$ respectively.
Hence, it follows from \eqref{phipreser} that $v_0+\bar{f_0}(x_0)\in\Hy^{k-r}\times\Sf^{l-s}\subset\Les^{k-r+1}\oplus\R^{l-s+1}$ for any $x_0\in M_0$.
Therefore, we have an isometric immersion $f_0\colon M_0\to\Hy^{k-r}\times\Sf^{l-s}$ such that \eqref{tildef} holds.
\qed

\begin{remark}
{\em
    A well known result by Nölker \cite{No} gives necessary and sufficient conditions for an isometric immersion of a warped product into a space-form to be an extrinsic warped product of immersions (see also Theorem $10.21$ in \cite{DT}). 
     Proposition \ref{propmulti} can be seen as an adaptation of Nölker's result to the particular case of hypersurfaces of $\Hy^k\times\Sf^l$. 
    In fact, the same arguments should work for submanifolds in higher codimension. 
    On the other hand, Schaaf generalized Nölker's result by proving a decomposition theorem for isometric immersions of warped products into semi-Riemannian warped products \cite{Sc}.
    Following \cite{Sc}, if we have a warped product representation of (an open subset of) $\Hy^k\times\Sf^l$, it suffices to check if the immersion and its second fundamental form are ``adapted" to this representation to conclude that it is an extrinsic warped product of immersions. It seemed easier, however, to follow the direct path in the proof of Proposition \ref{propmulti}.}
\end{remark}

\section{Weak congruence}\label{Sec:weak}

  In this section we introduce the notion of  $\emph{weak congruence}$ of hypersurfaces of 
  $\Hy^k\times\Sf^{n-k+1}$ in terms of the model of $\Hy^k\times\Sf^{n-k+1}$ as a submanifold of $\Les^{n+3}=\Les^{k+1}\oplus \R^{n-k+2}$. We characterize weak congruence of a pair of hypersurfaces in terms of their associated data $(A, R, S)$ as in Section $2$. This is then used to characterize weak congruence in terms of a certain flat bilinear form.

\begin{definition} \emph{ Let $f,g\colon M^n\to\Hy^{k}\times\Sf^{n-k+1}$ be isometric immersions, and let $\tilde{f}$ and $\tilde{g}$ be the corresponding immersions into $\Les^{n+3}=\Les^{k+1}\oplus \R^{n-k+2}$. We say that $f$ and $g$ are \emph{weakly congruent} if there exists an isometry $\Phi\colon \Les^{n+3}\to \Les^{n+3}$ that does not preserve $\Hy^{k}\times\Sf^{n-k+1}$ such that $\tilde g=\Phi\circ \tilde f$.}
\end{definition}

Using the conformal diffeomorphism between $\Hy^{k}\times\Sf^{n-k+1}$ and $\R^{n+1}\setminus\R^{k-1}$,  two isometric hypersurfaces in $\Hy^{k}\times\Sf^{n-k+1}$ can be seen as a pair of conformal hypersurfaces in $(\R^{n+1}\setminus\R^{k-1}, g_0)$, where $g_0$ is the Euclidean metric.
The hypersurfaces $f,g\colon M^n\to\Hy^{k}\times\Sf^{n-k+1}$ are weakly congruent if and only if the corresponding hypersurfaces in $(\R^{n+1}\setminus\R^{k-1}, g_0)$ are conformally congruent in $\R^{n+1}$.
On the other hand, the hypersurfaces $f$ and $g$ are congruent in $\Hy^{k}\times\Sf^{n-k+1}$ if and only if, when seen in $\R^{n+1}\setminus\R^{k-1}$, they are conformally congruent via a Moebius transformation of $\R^{n+1}$ that leaves $\R^{k-1}$ invariant. 

\begin{example}\label{weak} \emph{It was shown in Theorem $2$ of \cite{Ni} (see also Proposition $8$ of \cite{JT}) that any  umbilic submanifold of  $\Hy^{k}\times\Sf^{n-k+1}\subset \Les^{k+1}\oplus \R^{n-k+2}= \Les^{n+3}$ with codimension $p$ is the intersection of $\Hy^{k}\times\Sf^{n-k+1}$ with a 
time-like subspace of $\Les^{n+3}$ with codimension $p$. In particular, any umbilic hypersurface of  $\Hy^{k}\times\Sf^{n-k+1}\subset \Les^{k+1}\oplus \R^{n-k+2}= \Les^{n+3}$ is the intersection of $\Hy^{k}\times\Sf^{n-k+1}$ with a
time-like hyperplane of $\Les^{n+3}$. Given two time-like subspaces $V$ and $W$ of
$\Les^{n+3}$ of a given dimension, a necessary and sufficient condition for the existence of an isometry $\Phi$ of $\Les^{n+3}$ that preserves $\Les^{k+1}$ and $\R^{n-k+2}$ and takes one of the subspaces onto the other was obtained in Proposition $10$ of \cite{JT}. In particular, given two time-like hyperplanes $V$ and $W$ of $\Les^{n+3}$, this gives a necessary and sufficient condition for the corresponding umbilical hypersurfaces 
$V\cap(\Hy^{k}\times\Sf^{n-k+1})$ and $W\cap(\Hy^{k}\times\Sf^{n-k+1})$ to be congruent.}  
\end{example}

The next result characterizes weak congruence of two hypersurfaces in $\Hy^{k}\times\Sf^{n-k+1}$ in terms of their associated data.

\begin{proposition}\label{congL}
The oriented hypersurfaces  $f,g\colon M^n\to\Hy^{k}\times\Sf^{n-k+1}$  
are weakly congruent if and only if there exists  $\rho\in C^{\infty}(M)$ such that either

\begin{align}
    A_g&=A_f+\rho I\label{cong1}\\
    R_g&=R_f-\rho A_f-\frac{\rho^2}{2}I\label{cong2}\\
    S_g&=S_f+d\rho\label{cong3},
\end{align}
or 
\begin{align*}
    A_g&=-A_f+\rho I\\
    R_g&=R_f+\rho A_f-\frac{\rho^2}{2}I\\
    S_g&=-S_f+d\rho,
\end{align*}
depending on the choice of the unit normal vector fields along $f$ and $g$.
\end{proposition}

\proof Let $\tilde{f}$ and $\tilde{g}$ be the corresponding immersions in $\Les^{n+3}$, and let $\Psi\in O_1(n+3)$ be such  that $\tilde{g}=\Psi\circ\tilde{f}$. Then 
$$
\a_{\tilde{g}}(X,Y)=\Psi\a_{\tilde{f}}(X,Y)
$$
for all $X,Y\in\mathfrak{X}(M)$. Hence, using \eqref{sfft} we obtain
\begin{align}
    \<A_gX,Y\>&=\<\a_{\tilde{g}}(X,Y), i_*N_g\>\nonumber\\
    &=\<\Psi\a_{\tilde{f}}(X,Y), i_*N_g\>\nonumber\\
    &=\<\Psi(\<A_fX,Y\>i_*N_f+\<X,Y\>f_1-\<R_fX,Y\>\tilde{f}),i_*N_g\>\nonumber\\
    &=\<A_fX,Y\>\<\Psi i_*N_f,i_*N_g\>+\<X,Y\>\<\Psi f_1,i_*N_g\>-\<R_fX,Y\>\<\Psi\tilde{f},i_*N_g\>\nonumber\\
    &=\<A_fX,Y\>\<\Psi i_*N_f,i_*N_g\>+\<X,Y\>\<\Psi f_1,i_*N_g\>,\label{afg1}
\end{align}
where in the last step we have used that $\Psi\tilde{f}=\tilde{g}$ and that $\<i_*N_g,\tilde{g}\>=0$.

The normal bundle of $\tilde{g}$ is spanned by the orthonormal frame $\{i_*N_g,g_1,g_2\}$. Writing
$\Psi i_*N_f=ai_*N_g+bg_1+cg_2$, on the one hand, we have
\begin{align*}
    0&=\<\tilde{f},i_*N_f\>=\<\Psi\tilde{f},\Psi i_*N_f\>\\
    &=\<\tilde{g},ai_*N_g+bg_1+cg_2\>\\
    &=-b+c.
\end{align*}
On the other hand,
\begin{align*}
    1&=\<N_f,N_f\>=\<\Psi i_*N_f,\Psi i_*N_f\>\\
    &=a^2-b^2+c^2.
\end{align*}
We obtain from the previous relations that $b=c$ and $a^2=1$. We argue for the case $a=\<\Psi i_*N_f,i_*N_g\>=1$, the other case being similar. Hence $\Psi i_*N_f=i_*N_g+bg_1+bg_2$.
Denoting $\rho=\<\Psi f_1,i_*N_g\>$, Eq. \eqref{afg1} reduces to \eqref{cong1}.

From $\Psi\a_{\tilde{f}}=\a_{\tilde{g}}$ and \eqref{sfft} we have
\begin{align*}
-\<R_g X,Y\>&=\<\a_{\tilde{g}}(X,Y),g_2\>\\
&=\<A_fX,Y\>\<\Psi i_*N_f,g_2\>+\<X,Y\>\<\Psi f_1,g_2\>-\<R_fX,Y\>\<\Psi \tilde{f},g_2\>\\
&=b\<A_fX,Y\>+\<X,Y\>\<\Psi f_1,g_2\>-\<R_fX,Y\>
\end{align*}
for all $X,Y\in\mathfrak{X}(M)$. Hence
\be\label{rfg1}
R_g=R_f -bA_f-\<\Psi f_1,g_2\>I.
\ee
Write $\Psi f_1=\rho i_*N_g+\phi_1 g_1+\phi_2g_2$. Then 
\begin{align*}
    0&=\<\Psi f_1,\Psi i_*N_f\>=\<\rho i_*N_g+\phi_1 g_1+\phi_2g_2,i_*N_g+bg_1+bg_2\>\\
    &=\rho-b\phi_1+b\phi_2,
\end{align*}
thus $b(\phi_1-\phi_2)=\rho$.
Similarly, from $\<\Psi f_1,\Psi f_1\>=-1$ we obtain
\be\label{psis1}
\rho^2-\phi_1^2+\phi_2^2=-1.
\ee
Notice that $\Psi f_2=\tilde{g}-\Psi f_1$, so $\<\Psi f_2,\Psi f_2\>=1$ gives
\begin{align*}
1&=\<\tilde{g}-\Psi f_1,\tilde{g}-\Psi f_1\>\\
&=-2\<\tilde{g},\Psi f_1\>-1\\
&=-1+2(\phi_1-\phi_2).
\end{align*}
Therefore, 
$$
\phi_1-\phi_2=1,
$$
and hence $\rho=b$.
Finally, replacing the previous relations in \eqref{psis1}
yields $$\phi_1=1+\frac{\rho^2}{2}\,\,\,\mbox{and} \,\,\,\phi_2=\frac{\rho^2}{2}.$$
Thus $\Psi f_1=\rho i_*N_g+(1+\frac{\rho^2}{2})g_1+\frac{\rho^2}{2}g_2$, and 
\eqref{rfg1} becomes \eqref{cong2}.

Let $\nab$ be the Levi-Civita connection of $\Les^{n+3}$ and let $\nabla^{\tilde{f}\perp}$ be the normal connection of $\tilde{f}$. Then we have from \eqref{L} that
\begin{align*}
\nabla^{\tilde{f}\perp}_Xi_*N_f&=(\<f_*X,N_f\>-\<\pi_2f_*X,\pi_2N_f\>)f_1-\<\pi_2f_*X,\pi_2N_f\>f_2\\
&=-S_f(X)f_1-S_f(X)f_2\\
&=-S_f(X)\tilde{f},    
\end{align*}
and similarly,
$$
\nabla^{\tilde{g}\perp}_Xi_*N_g=-S_g(X)\tilde{g} 
$$
for any $X\in\mathfrak{X}(M)$.
Since $\Psi$ is constant, then  $\nab_X\Psi i_*N_f=\Psi\nab_Xi_*N_f$. 
On one hand, we have
\begin{align*}
(\nab_X\Psi i_*N_f)_{N_{\tilde{g}}M}&=(\nab_X(i_*N_g+\rho\tilde{g})))_{N_{\tilde{g}}M}\\
&=(-S_g(X)+X(\rho))\tilde{g}.
\end{align*}
On the other hand,
$$
(\Psi\nab_Xi_*N_f)_{N_{\tilde{g}}M}=-S_f(X)\Psi\tilde{f}=-S_f(X)\tilde{g}
$$
which implies \eqref{cong3}. The case $a=-1$ leads to the second system of equations in the statement.

Conversely, assume that \eqref{cong1}, \eqref{cong2} and \eqref{cong3} hold (the case of the second system of equations follows with a similar argument).
Define $\Psi\in \Gamma(\Hom(N_{\tilde{f}}M,N_{\tilde{g}}M))$ by
\begin{align*}
    \Psi i_*N_f&=i_*N_g+\rho g_1+\rho g_2\\
    \Psi f_1&= \rho i_*N_g+\left(1+\frac{\rho^2}{2}\right)g_1+\frac{\rho^2}{2}g_2\\
    \Psi f_2&= -\rho i_*N_g-\frac{\rho^2}{2}g_1+\left(1-\frac{\rho^2}{2}\right)g_2.
\end{align*}
A straightforward computation shows that $\Psi$ is a vector bundle isometry between the normal bundles of $\tilde{f}$ and $\tilde{g}$. Notice that $\Psi$ satisfies $\Psi\tilde{f}=\tilde{g}$.
Using \eqref{sfft}, \eqref{cong1} and \eqref{cong2}, we have
\begin{align*}
    \Psi\a_{\tilde{f}}(X,Y)=&\Psi(\<A_fX,Y\>i_*N_f+\<X,Y\>f_1-\<R_fX,Y\>\tilde{f})\\
    =&\<A_fX,Y\>(i_*N_g+\rho \tilde{g})+\<X,Y\>\left(\rho i_*N_g+g_1+\frac{\rho^2}{2}\tilde{g}\right)\\
    &-\<R_fX,Y\>\tilde{g}\\
    =&(\<A_fX,Y\>+\rho\<X,Y\>)i_*N_g+\<X,Y\>g_1\\
    &+\left(\rho\<A_fX,Y\>+\frac{\rho^2}{2}\<X,Y\>-\<R_fX,Y\>\right)\tilde{g}\\
    =&\<A_gX,Y\>i_*N_g+\<X,Y\>g_1-\<R_gX,Y\>\tilde{g}\\
    =&\a_{\tilde{g}}(X,Y)
\end{align*}
for all $X,Y\in\mathfrak{X}(M)$.

The normal connection of $\tilde{f}$ is given by
$$    \nabla^{\tilde{f}\perp}_Xi_*N_f=-S_f(X)\tilde{f}, \,\,\,\,    \nabla^{\tilde{f}\perp}_Xf_1=-S_f(X)i_*N_f\,\,\,\mbox{and}\,\,\,
    \nabla^{\tilde{f}\perp}_Xf_2=S_f(X)i_*N_f
    $$
for any $X\in\mathfrak{X}(M)$.
The normal connection of $\tilde{g}$ is given in an analogous way.
We claim that $\Psi$ is parallel. Indeed, we have
$$
    \nabla^{\tilde{g}\perp}_X\Psi i_*N_f=\nabla^{\tilde{g}\perp}_X(i_*N_g+\rho \tilde{g})=(-S_g(X)+X(\rho))\tilde{g}
$$
 and 
 $$
\Psi\nabla^{\tilde{f}\perp}_Xi_*N_f=\Psi(-S_f(X)\tilde{f})=-S_f(X)\tilde{g}.
$$
Then \eqref{cong3} implies that $\nabla^{\tilde{g}\perp}_X\Psi i_*N_f=\Psi\nabla^{\tilde{f}\perp}_Xi_*N_f$. Similarly, straightforward computations using \eqref{cong3} give 
$$\nabla^{\tilde{g}\perp}_X\Psi f_1=\Psi\nabla^{\tilde{f}\perp}_Xf_1\,\,\,\mbox{and}\,\,\,\nabla^{\tilde{g}\perp}_X\Psi f_2=\Psi\nabla^{\tilde{f}\perp}_Xf_2,$$ thus proving the claim.
Hence, extending $\Psi$ by $\Psi\tilde{f}_*X=\tilde{g}_*X$ for $X\in\mathfrak{X}(M)$, we see that $\Psi$ is in fact a constant linear isometry of $\Les^{n+3}$ such that $\Psi\tilde{f}=\tilde{g}$.
\qed
\vspace{2ex}

\section{A flat bilinear form}\label{Sec:flatbil}

In this section, we show that a pair of isometric immersions $f,g\colon M^n \to \mathbb{H}^k\times \mathbb{S}^{n-k+1}$ gives rise to a certain flat bilinear form, and we characterize in terms of it when $f$ and $g$ are weakly congruent.\vspace{1ex}

Given a vector space $W^{p,q}$  of dimension $p+q$  endowed with an  
inner product $\<\!\<\,,\,\>\!\>$ of signature $(p,q)$, and 
finite dimensional vector spaces $V$,  $U$, recall that a  bilinear form 
$\beta\colon V\times U\to W^{p,q}$ is said to be 
\emph{flat} with respect to $\<\!\<\,,\,\>\!\>$ if 
$$
\<\!\<\,\beta(X,W),\beta(Y,Z)\>\!\>-\<\!\<\,\beta(X,Z),\beta(Y,W)\>\!\>=0
$$
for all  $X,Y\in V$ and $Z,W\in U$. It is called
\emph{null} if
$$
\<\!\<\,\beta(X,W),\beta(Y,Z)\>\!\>=0
$$
for all  $X,Y\in V$ and $Z,W\in U$. Thus a null bilinear form is
necessarily flat.

Let $f,g\colon M^n\to \Hy^{k}\times\Sf^{n-k+1}$, $2\leq k\leq n-1$, be isometric immersions and consider $\R^2\oplus\R^2$ endowed with the indefinite inner product $\<\!\<\,,\,\>\!\>$ of signature $(2,2)$. 
Define, at each $x\in M^n$,  the bilinear form $\beta\colon T_xM\times T_xM\to \R^{2,2}$ by
\be\label{beta}
\beta(X,Y)=(\<A_fX,Y\>, \<\psi_+X,Y\>,\<A_gX,Y\>,\<\psi_-X,Y\>),
\ee
where $\psi_{\pm}= \frac{1}{\sqrt{2}}(R_f-R_g\pm I)$.
\begin{proposition}\label{prop:flat} 
    The bilinear form $\beta$ is flat with respect to $\<\!\<\,,\,\>\!\>$.
\end{proposition}
\proof
A straightforward computation gives
\begin{align*}
    \<\!\<\,\beta(X,W)&,\beta(Y,Z)\,\>\!\>-\<\!\<\,\beta(X,Z),\beta(Y,W)\,\>\!\>\\
    =&\<(A_fX\wedge A_fY)Z,W\>+\<(X\wedge R_fY)Z,W\>+\<(R_fX\wedge Y)Z,W\>\\
    &-\<(A_gX\wedge A_gY)Z,W\>-\<(X\wedge R_gY)Z,W\>-\<(R_gX\wedge Y)Z,W\>\\
    =&0,
\end{align*}
where the last step follows from the Gauss equation \eqref{Gauss} for $f$ and $g$.\vspace{2ex}
\qed

The \emph{relative nullity} subspace at $x\in M^n$ of an isometric immersion $f\colon M^n\to\Hy^k\times\Sf^{n-k+1}$ is the kernel of its second fundamental form at $x$, whose dimension, denoted by $\nu_f(x)$, is the \emph{index of relative nullity} of $f$ at $x$.

\begin{proposition}\label{prop:null}
   Let $f,g\colon M^n\to \Hy^{k}\times\Sf^{n-k+1}$, $2\leq k\leq n-1$, be isometric immersions and let $\beta$ be the associated flat bilinear form. If  $f$ and $g$ are weakly congruent, then $\beta$ is null. 
   Conversely, if $\beta$ is null and one of the immersions has index of relative nullity less than or equal to $n-2$ at every point, then $f$ and $g$ are weakly congruent.
      
\end{proposition}

\proof
The equation
$$
\<\!\<\,\beta(X,W),\beta(Y,Z)\>\!\>=0
$$
for all $X,Y,Z,W\in\mathfrak{X}(M)$ is equivalent to
\be\label{betanull}
\<A_fY,Z\>A_f-\<A_gY,Z\>A_g+\<(R_f-R_g)Y,Z\>I+\<Y,Z\>(R_f-R_g)=0
\ee
for all $Y,Z\in\mathfrak{X}(M)$. The preceding equation is satisfied if $A_g=A_f+\rho I$ and 
$R_f-R_g=\rho A_f+\frac{\rho^2}{2} I$ for some $\rho \in C^{\infty}(M)$, or if $A_g=-A_f+\rho I$ and $R_f-R_g=-\rho A_f+\frac{\rho^2}{2} I$. Thus $\beta$ is null if $f$ and $g$ are weakly congruent by Proposition \ref{congL}.

Conversely, assume that \eqref{betanull} is satisfied for all $Y,Z\in\mathfrak{X}(M)$ and that $\nu_g\leq n-2$ at any point. Choosing $Y$ and $Z$ as orthonormal principal directions of $f$, then \eqref{betanull} becomes
$$
-\<A_gY,Z\>A_g+\<(R_f-R_g)Y,Z\>I=0.
$$
If  $\<A_gY,Z\>\neq 0$, then the preceding equation implies that $A_g$ is a multiple of the identity, a contradiction. 
Therefore $A_g$ and $A_f$ commute, and from \eqref{betanull} they also commute with $R_f-R_g$. Let $\lambda_i, \mu_i$ and $r_i$, $1\leq i\leq n$, denote the eigenvalues of $A_f,A_g$ and $R_f-R_g$, respectively, corresponding to common eigenvectors $X_i$. Then
\be\label{betanulli}
\lambda_iA_f-\mu_iA_g+r_iI+(R_f-R_g)=0.
\ee
First assume that $f$ is umbilical, that is $A_f=\lambda I$ for some $\lambda\in C^{\infty}(M)$. Then the above equation becomes
$$
\lambda^2 I-\mu_iA_g+r_iI+(R_f-R_g)=0.
$$
If $A_g$ has two distinct eigenvalues $\mu_i\neq\mu_j$, it follows that
$$
(\mu_j-\mu_i)A_g+(r_i-r_j)I=0,
$$
showing that $A_g$ is a multiple of the identity, a contradiction. Then $g$ must also be umbilical, $A_g=\mu I$, and consequently $R_f-R_g=rI$.
Eq. \eqref{betanulli} gives $\lambda^2-\mu^2+2r=0$ and, writing $\rho=\mu-\lambda$, we see that \eqref{cong1} and \eqref{cong2} hold.

Now assume that neither $f$ nor $g$ is umbilical. Then there are at least two distinct eigenvalues $\lambda_i\neq\lambda_j$ of $A_f$. Taking the difference of \eqref{betanulli} for $i$ and $j$ we obtain
$$
(\lambda_i-\lambda_j)A_f-(\mu_i-\mu_j)A_g+(r_i-r_j)I=0,
$$
which implies that
$$
(\lambda_i-\lambda_j)^2=(\mu_i-\mu_j)^2.
$$
Changing the orientation of $g$, if necessary, we can assume that $\lambda_i-\lambda_j=\mu_i-\mu_j$, and hence \eqref{cong1} holds for $\rho=(r_i-r_j)/(\lambda_i-\lambda_j)$.
Therefore, replacing \eqref{cong1} in \eqref{betanulli} gives
\begin{align*}
   0&=\lambda_iA_f-(\lambda_i+\rho)(A_f+\rho I)+r_iI+(R_f-R_g)\\
&=-\lambda_i\rho I-\rho A_f-\rho^2I+\frac{1}{2}(\mu_i^2-\lambda_i^2)I+(R_f-R_g)\\
&=-\lambda_i\rho I-\rho A_f-\rho^2I+\frac{1}{2}((\lambda_i+\rho)^2-\lambda_i^2)I+(R_f-R_g)\\
&=-\rho A_f-\frac{\rho^2}{2}I+(R_f-R_g),
\end{align*}
which is \eqref{cong2}.
Hence \eqref{cong1} and \eqref{cong2} hold in any case, and Proposition \ref{congL} together with Lemma \ref{12imp3} below imply that $f$ and $g$ are weakly congruent.\vspace{1ex}
\qed

\begin{lemma}\label{12imp3}
Let $f,g\colon M^n\to\Hy^{k}\times\Sf^{n-k+1}$, $2\leq k\leq n-1$, be isometric immersions. If $f$ or $g$  has  
index of relative nullity less than or equal to $n-2$ at any point, then equations \eqref{cong1} and \eqref{cong2} imply \eqref{cong3}.
\end{lemma}

\proof Assume, without loss of generality, that $\nu_g\leq n-2$ at any point.
 Using \eqref{derR} for $g$, we have 
\be\label{codR}
(\nabla_XR_g)Y-(\nabla_YR_g)X=S_g(Y)A_gX-S_g(X)A_gY.
\ee
On the other hand, from \eqref{cong2}  we have 
\begin{align*}
(\nabla_XR_g)Y&=(\nabla_X(R_f-\rho A_f-\frac{\rho^2}{2}I))Y\\
&=(\nabla_XR_f)Y-X(\rho)A_fY-\rho(\nabla_XA_f)Y-\rho X(\rho)Y.
\end{align*}
Hence
\begin{align*}
(\nabla_XR_g)Y-(\nabla_YR_g)X=&(\nabla_XR_f)Y-X(\rho)A_fY-\rho(\nabla_XA_f)Y-\rho X(\rho)Y\\
&-(\nabla_YR_f)X+Y(\rho)A_fX+\rho(\nabla_YA_f)X+\rho Y(\rho)X.
\end{align*}
Using \eqref{derR}, \eqref{codazzi} and \eqref{cong1}, we obtain
\begin{align*}
    (\nabla_XR_g)Y-(&\nabla_YR_g)X\\
    =&S_f(Y)A_fX-S_f(X)A_fY-X(\rho)A_fY+Y(\rho)A_fX\\
    &+\rho S_f(Y) X-\rho S_f(X)Y-\rho X(\rho)Y+\rho Y(\rho)X\\
    =&(S_f(Y)+Y(\rho))A_fX-(S_f(X)+X(\rho))A_fY\\
    &+\rho(S_f(Y)+Y(\rho))X-\rho(S_f(X)+X(\rho))Y\\
    =&(S_f(Y)+Y(\rho))A_gX-(S_f(X)+X(\rho))A_gY.
\end{align*}
Replacing \eqref{codR} into the preceding equation yields
$$
S_g(Y)A_gX-S_g(X)A_gY=(S_f(Y)+Y(\rho))A_gX-(S_f(X)+X(\rho))A_gY
$$
 for all $X,Y\in\mathfrak{X}(M)$. Given  $X\in\mathfrak{X}(M)$, at each point $x\in M^n$ where  $A_gX\neq 0$  one can choose $Y\in T_xM$ such that  $A_gX$ and  $A_gY$ are linearly independent,  by the assumption that $\nu_g\leq n-2$ at any point. Then the above equation implies that  $S_g(X)=S_f(X)+X(\rho)$ at $x$. On the other hand, if $A_gX= 0$ at $x$, choosing any $Y\in T_xM$ such that  $A_gY\neq 0$, the above equation shows that also in this case  $S_g(X)=S_f(X)+X(\rho)$ at $x$.  
\qed
\vspace{2ex}

   The notions of congruence and weak congruence give rise to distinct notions of rigidity and deformability.

\begin{definition} \emph{ Let $f\colon M^n\to\Hy^{k}\times\Sf^{n-k+1}$ be an isometric immersion. We say that  $f$ is:
\begin{itemize}
\item[(i)] \emph{rigid}, if any other isometric immersion $g\colon M^n\to\Hy^{k}\times\Sf^{n-k+1}$ is congruent to $f$;
\item[(ii)] \emph{weakly rigid}, if any other isometric immersion $g\colon M^n\to\Hy^{k}\times\Sf^{n-k+1}$ is weakly congruent to $f$;
\item[(iii)] \emph{deformable}, if it is not rigid;
\item[(iv)]  \emph{truly deformable}, if there exists an  isometric immersion $g\colon M^n\to\Hy^{k}\times\Sf^{n-k+1}$ that is not weakly congruent to $f$ on any open subset. Any such isometric immersion is said to be a \emph{true deformation} of $f$.
\end{itemize}}
\end{definition} 

   In the next section, we study the hypersurfaces $f\colon M^n\to\Hy^{k}\times\Sf^{n-k+1}$ that admit deformations $g\colon M^n\to\Hy^{k}\times\Sf^{n-k+1}$ that are weakly congruent to $f$. The subsequent one is devoted to those that are truly deformable.

\section{Hypersurfaces with weakly congruent deformations} \label{Sec:weakdef}

   In this section, we investigate the  hypersurfaces $f\colon M^n\to\Hy^{k}\times\Sf^{n-k+1}$, $2\leq k\leq n-1$, that admit deformations $g\colon M^n\to\Hy^{k}\times\Sf^{n-k+1}$ that are weakly congruent to $f$. We say in short that $g$ is a 
   \emph{weakly congruent deformation} of~$f$.\vspace{1ex}

   Let $\Phi$ be a linear isometry of $ \Les^{n+3}$, and let $\{e_1, \ldots, e_{k+1}, e_{k+2}, \ldots, e_{n+3}\}$ 
be an orthonormal basis of 
$\Les^{n+3}$ such that $\{e_1, \ldots, e_{k+1}\}$ is an orthonormal basis of 
$\Les^{k+1}$ with $\<e_1, e_1\>=-1$ and $\{e_{k+2}, \ldots, e_{n+3}\}$ is an orthonormal basis of  $\R^{n-k+2}$.  Denote  $v_j=\Phi(e_j)=\sum_{i=1}^{n+3} a_{ij} e_i$, $1\leq j\leq n+3$, where $(a_{ij})$ is an orthogonal Lorentzian matrix.
Given $X=\sum_{j=1}^{n+3} x_j e_j$, we have 
$$
\Phi X=\sum_{j=1}^{n+3} x_j v_j=\sum_{j=1}^{n+3}x_j\left( \sum_{i=1}^{n+3} a_{ij}e_i \right)=\sum_{i=1}^{n+3}\left( \sum_{j=1}^{n+3}a_{ij} x_j \right)e_i.
$$

Define $F=F_\Phi=(F_1, F_2, F_3, F_4)\colon \Les^{n+3} \to \mathbb{R}^4$ by
$$
\left\{ \begin{array}{l}
F_1(X)=-x_1^2+\sum_{i=2}^{k+1} x_i^2;\vspace{1ex}\\
F_2(X)=\sum_{i=k+2}^{n+3} x_i^2;\vspace{1ex}\\
F_3(X)=-\left(\sum_{j=1}^{n+3}a_{1j}x_j\right)^2+\sum_{\ell=2}^{k+1} \left(\sum_{j=1}^{n+3}a_{\ell j}x_j\right)^2;\vspace{1ex}\\
F_4(X)=\sum_{\ell=k+2}^{n+3}\left(\sum_{j=1}^{n+3}a_{\ell j}x_j\right)^2.
\end{array}
\right.
$$

\begin{proposition}\label{dif} Let $\Phi$ be an isometry of $\Les^{n+3}=\Les^{k+1}\oplus \R^{n-k+2}$ that does not preserve $\Les^{k+1}$ and $\R^{n-k+2}$. Assume that $M=F_\Phi^{-1}(-1,1, -1,1)\subset \Les^{n+3}$ is nonempty and that $F_\Phi$ has rank $3$ on an open neighborhood of $M$. Then $M$ is a hypersurface of dimension $n$ of $\Hy^{k}\times\Sf^{n-k+1}$ and the inclusion $j\colon M\to \Hy^{k}\times\Sf^{n-k+1}$ is deformable, with $\Phi\circ j$ a weakly congruent deformation of $j$.

Conversely, if $f\colon M\to \Hy^{k}\times\Sf^{n-k+1}$ is a hypersurface that admits a weakly congruent deformation $g\colon M\to \Hy^{k}\times\Sf^{n-k+1}$, then there exists an isometry $\Phi$ of $\Les^{n+3}$ that does not preserve $\Hy^{k}\times\Sf^{n-k+1}$ such that $f(M)\subset F_\Phi^{-1}(-1,1,-1,1)$.
\end{proposition}
\proof That $M$ is a submanifold of dimension $n$ of $\Les^{n+3}$ follows from the fact that $F$ has rank $3$ in an open neighborhood of $M$.  Since $M\subset F_1^{-1}(-1)\cap  F_2^{-1}(1)$, then $M\subset \Hy^{k}\times\Sf^{n-k+1}$. The fact that  $M\subset F_3^{-1}(-1)\cap  F_4^{-1}(1)$ implies that $\Phi(M)\subset 
\Hy^{k}\times\Sf^{n-k+1}$. Hence $\Phi\circ j\colon M\to \Hy^{k}\times\Sf^{n-k+1}$ is an isometric immersion that is not congruent to the inclusion $j\colon M\to \Hy^{k}\times\Sf^{n-k+1}$, for $\Phi$ does not preserve $\Hy^{k}\times\Sf^{n-k+1}$. Thus $j\colon M\to \Hy^{k}\times\Sf^{n-k+1}$ is deformable, with $\Phi\circ j$ a weakly congruent deformation of $j$.

Conversely, if $f\colon M\to \Hy^{k}\times\Sf^{n-k+1}$ admits a weakly congruent deformation $g\colon M\to \Hy^{k}\times\Sf^{n-k+1}$, denote by $\tilde f$ and $\tilde g$ the corresponding immersions into $\Les^{n+3}$, and let $\Phi$ be an isometry of $\Les^{n+3}$ that does not preserve $\Hy^{k}\times\Sf^{n-k+1}$ such that $\tilde g=\Phi\circ \tilde f$. 
Then $f(M)\subset F_\Phi^{-1}(-1,1,-1,1)$.  \qed

\begin{examples}\label{exaweak}
$(i)$ 
\emph{Let $\{e_1, \ldots, e_{k+1}, e_{k+2}, \ldots, e_{n+3}\}$ be an orthonormal basis of $\Les^{n+3}=\Les^{k+1}\oplus \R^{n-k+2}$ such that $\{e_1, \ldots, e_{k+1}\}$ is an orthonormal basis of 
$\Les^{k+1}$ with $\<e_1, e_1\>=-1$ and $\{e_{k+2}, \ldots, e_{n+3}\}$ is an orthonormal basis of  $\R^{n-k+2}$. 
We write any $x\in\Les^{n+3}$ as $x=\sum_{i=1}^{n+3}x_ie_i$.
For a fixed $\theta\in(0,\pi)$, let $\Phi \in O_1(n+3)$ be given by 
$$
\Phi(e_2)=\cos(\theta)e_2+\sin(\theta)e_{k+2}, \,\,\,\Phi(e_{k+2})=-\sin(\theta)e_2+\cos(\theta)e_{k+2}
$$
and $\Phi(e_i)=e_i$ if $1\leq i\leq n+3$, $i\neq 2$, $i\neq k+2$. Then $\Phi$ does not preserve the orthogonal decomposition $\Les^{n+3}=\Les^{k+1}\oplus \R^{n-k+2}$.
Let $F_{\Phi}$ be given as in Proposition \ref{dif}.
The subset $M=F_{\Phi}^{-1}(-1,1, -1,1)$, consists of the solutions of the system of equations
$$
\begin{cases}
    -1&=-x_1^2+\sum_{i=2}^{k+1}x_i^2\\
    1&=\sum_{j=k+2}^{n+3}x_j^2\\
    -1&=-x_1^2+(\cos(\theta)x_2-\sin(\theta)x_{k+2})^2+\sum_{i=3}^{k+1}x_i^2\\
    1&=(\sin(\theta)x_2+\cos(\theta)x_{k+2})^2+\sum_{j=k+3}^{n+3}x_j^2,
\end{cases}
$$
which is equivalent to 
$$
\begin{cases}
    -1&=-x_1^2+\sum_{i=2}^{k+1}x_i^2\\
    1&=\sum_{j=k+2}^{n+3}x_j^2\\
    0&=\sin^2(\theta)(x_2^2-x_{k+2}^2)+2\cos(\theta)\sin(\theta)x_2x_{k+2}.
\end{cases}
$$
Since $\theta\in(0,\pi)$, we can write the last equation as
$$
x_2^2+ax_2x_{k+2}-x_{k+2}^2=0,
$$
where $a=2\cot(\theta)$, or equivalently,
$$
\left(\frac{\sqrt{a^2+4}-a}{2}x_2+x_{k+2}\right)\left(\frac{\sqrt{a^2+4}+a}{2}x_2-x_{k+2}\right)=0.
$$
Thus $F_{\Phi}^{-1}(-1,1, -1,1)\subset V\cup W$, where $V$ and $W$ are the orthogonal complements of the space-like vectors 
$$
v=\frac{\sqrt{a^2+4}-a}{2}e_2+e_{k+2}\,\,\,\,\mbox{and}\,\,\,\,w=\frac{\sqrt{a^2+4}+a}{2}e_2-e_{k+2},
$$ 
respectively.
The umbilical hypersurfaces $M_1=V\cap(\Hy^k\times\Sf^{n-k+1})$ and $M_2=W\cap(\Hy^k\times\Sf^{n-k+1})$ are such that $\Phi(M_i)\subset\Hy^k\times\Sf^{n-k+1}$, $i=1,2$. Therefore, the inclusion
$j_i\colon M_i\to \Hy^k\times\Sf^{n-k+1}$, $i=1,2$, is deformable, with $\Phi\circ j_i$  a weakly congruent deformation of it. Notice that $\Phi\circ j_i$, $i=1,2$, is also umbilical, with
$\Phi(M_1)\cup \Phi(M_2)=F_{\Phi^{-1}}^{-1}(-1,1,-1,1)$.
\vspace{1ex}\\
$(ii)$ Let $\{e_1, \ldots, e_{k+1}, e_{k+2}, \ldots, e_{n+3}\}$ be an orthonormal basis of $\Les^{n+3}=\Les^{k+1}\oplus \R^{n-k+2}$ such that $\{e_1, \ldots, e_{k+1}\}$ is an orthonormal basis of 
$\Les^{k+1}$ as before.
For a fixed $\theta\in\R\setminus\{0\}$, define $\Phi \in O_1(n+3)$ by 
$$
\Phi(e_1)=\cosh(\theta)e_1+\sinh(\theta)e_{k+2}, \,\,\,\,\Phi(e_{k+2})=\sinh(\theta)e_1+\cosh(\theta)e_{k+2}
$$ 
and $\Phi(e_i)=e_i$ for $2\leq i\leq n+3$, $i\neq k+2$. 
Then $\Phi$ does not preserve the orthogonal decomposition $\Les^{n+3}=\Les^{k+1}\oplus \R^{n-k+2}$.
The subset $M=F_{\Phi}^{-1}(-1,1, -1,1)$ consists of the solutions to the system of equations
$$
\begin{cases}
    -1&=-x_1^2+\sum_{i=2}^{k+1}x_i^2\\
    1&=\sum_{j=k+2}^{n+3}x_j^2\\
    -1&=-(\cosh(\theta)x_1+\sinh(\theta)x_{k+2})^2+\sum_{i=2}^{k+1}x_i^2\\
    1&=(\sinh(\theta)x_1+\cosh(\theta)x_{k+2})^2+\sum_{j=k+3}^{n+3}x_j^2,
\end{cases}
$$
which is equivalent to 
$$
\begin{cases}
    -1&=-x_1^2+\sum_{i=2}^{k+1}x_i^2\\
    1&=\sum_{j=k+2}^{n+3}x_j^2\\
    0&=\sinh^2(\theta)(x_1^2+x_{k+2}^2)+2\cosh(\theta)\sinh(\theta)x_1x_{k+2}.
   \end{cases}
$$
Since $\theta\neq 0$, we can write the last equation as
$$
x_1^2+ax_2x_{k+2}+x_{k+2}^2=0,
$$
where $a=2\coth(\theta)$, or equivalently,
$$
\left(\frac{a+\sqrt{a^2-4}}{2}x_1+x_{k+2}\right)\left(\frac{a-\sqrt{a^2-4}}{2}x_1+x_{k+2}\right)=0.
$$
Notice that $|a|=|2\coth(\theta)|>2$.
The preceding equation can also be written as
$
\langle X, v\rangle \langle X, w\rangle=0$
for $X=\sum_{j=1}^{n+3} x_j e_j$, 
$$
v=-\frac{a+\sqrt{a^2-4}}{2}e_1+e_{k+2}
\,\,\,\,\mbox{and}\,\,\,\,
w=\frac{\sqrt{a^2-4}-a}{2}e_1+e_{k+2}.
$$
If $-\infty<a<-2$, then $v$ is space-like and $w$ is time-like, whereas $w$ is space-like and $v$ is time-like if $2<a<\infty$. Since $X\in \Hy^k\times\Sf^{n-k+1}$, then $\langle X, X\rangle=0$, hence $X$ can not be orthogonal to a time-like vector. Therefore
 $M=V\cap(\Hy^k\times\Sf^{n-k+1})$ if $-\infty<a<-2$, and $M=W\cap(\Hy^k\times\Sf^{n-k+1})$ if $2<a<\infty$, where $V$ and $W$ are the orthogonal complements of $v$ and $w$, respectively. Thus $M$ is an umbilical hypersurface in either case. 
A straightforward computation shows that
$$
\Phi(v)=-\sinh(\theta)\frac{\sqrt{a^2-4}}{2}\left(\frac{a+\sqrt{a^2-4}}{2}e_1+e_{k+2}\right)
$$
and 
$$
\Phi(w)=\sinh(\theta)\frac{\sqrt{a^2-4}}{2}\left(\frac{a-\sqrt{a^2-4}}{2}e_1+e_{k+2}\right).
$$
Calling $v'=\frac{a+\sqrt{a^2-4}}{2}e_1+e_{k+2}$ and $w'=\frac{a-\sqrt{a^2-4}}{2}e_1+e_{k+2}$, and denoting by $V'$ and $W'$  their respective orthogonal complements, we conclude that the inclusion $j\colon M\to \Hy^k\times\Sf^{n-k+1}$ is deformable, with either $M'= V'\cap(\Hy^k\times\Sf^{n-k+1})$ or $M'= W'\cap(\Hy^k\times\Sf^{n-k+1})$ as a weakly congruent deformation, according to whether $-\infty<a<-2$ or $2<a<\infty$, respectively. 
\vspace{1ex}\\
$(iii)$ Let
 $\{e_1, \ldots, e_{k+1}, e_{k+2}, \ldots, e_{n+3}\}$ be a pseudo-orthonormal basis of $\Les^{n+3}$ such that $\{e_1, \ldots, e_{k+1}\}$ is a pseudo-orthonormal basis of 
$\Les^{k+1}$ with $$\<e_1,e_1\>=0=\<e_2,e_2\>\,\,\,\mbox{and}\,\,\,\<e_1,e_2\>=1.$$
For a fixed $t\in\R\setminus\{0\}$, define $\Phi \in O_1(n+3)$ by $$\Phi(e_1)=e_1-\frac{t^2}{2}e_2+te_{k+2},\,\,\,\Phi(e_{k+2})=-te_2+e_{k+2}$$ and $\Phi(e_i)=e_i$ for $2\leq i\leq n+3$, $i\neq k+2$. 
Then $\Phi$ does not preserve the orthogonal decomposition $\Les^{n+3}=\Les^{k+1}\oplus \R^{n-k+2}$.
The subset $M=F_{\Phi}^{-1}(-1,1, -1,1)$ consists of the solutions of the system of equations
$$
\begin{cases}
    -1&=2x_1x_2+\sum_{i=3}^{k+1}x_i^2\\
    1&=\sum_{j=k+2}^{n+3}x_j^2\\
    -1&=2x_1(x_2-\frac{t^2}{2}x_1-tx_{k+2})+\sum_{i=3}^{k+1}x_i^2\\
    1&=(tx_1+x_{k+2})+\sum_{j=k+3}^{n+3}x_j^2,
\end{cases}
$$
which is equivalent to
\begin{equation}\label{syst}
\begin{cases}
    -1&=2x_1x_2+\sum_{i=3}^{k+1}x_i^2\\
    1&=\sum_{j=k+2}^{n+3}x_j^2\\
    0&=tx_1(\frac{t}{2}x_1+x_{k+2}).
    \end{cases}
\end{equation}
Since $t\neq 0$,  we can write the last equation as
$$
x_1\left(\frac{t}{2}x_1+x_{k+2}\right)=0.
$$
The first equation in \eqref{syst} implies that $x_1\neq 0$. Hence
$$
\frac{t}{2}x_1+x_{k+2}=0,
$$
which can be written as 
$\<X, v\>=0$,
where $X=\sum_{j=1}^{n+3} x_je_j$ and $v$ is the space-like vector
$$
v=\frac{t}{4}e_2+e_{k+2}.
$$
Thus $M=V\cap(\Hy^k\times\Sf^{n-k+1})$, where $V=\{v\}^\perp$, and the inclusion $j\colon M \to \Hy^k\times\Sf^{n-k+1}$ is a deformable hypersurface, with a weakly congruent deformation given by the umbilical hypersurface $M=V'\cap(\Hy^k\times\Sf^{n-k+1})$, where $V'$ is the orthogonal complement of $v'=\Phi(v)=-\frac{3t}{4}e_2+e_{k+2}$.}
\end{examples}

We now combine the examples above to prove the following.

\begin{proposition}\label{defumb}
    Any non-totally geodesic umbilical hypersurface of $\Hy^k\times\Sf^{n-k+1}$, $2\leq k\leq n-1$, admits a weakly congruent umbilical deformation.
\end{proposition}
\proof  Let $M$ be a non-totally geodesic umbilical hypersurface of $\Hy^k\times\Sf^{n-k+1}$, $2\leq k\leq n-1$. As pointed out in Example \ref{weak}, there exists a unit-length space-like vector $v\in\Les^{n+3}=\Les^{k+1}\oplus \R^{n-k+2}$ such that $M=V\cap (\Hy^k\times\Sf^{n-k+1})$ 
of $\Hy^k\times\Sf^{n-k+1}$, where $V=\{v\}^\perp$.  
 Write $v=v_1+v_2$ according to the orthogonal decomposition $\Les^{n+3}=\Les^{k+1}\oplus \R^{n-k+2}$. Since  $M$ is not totally geodesic, then $v_1\neq 0\neq v_2$ (see Corollary $5$ of \cite{JT}). We consider separately the three possible cases: \vspace{1ex}\\
 $(i)$ $v_1$ is space-like: 
Fix an orthonormal basis $\{e_1, \ldots, e_{k+1}, e_{k+2}, \ldots, e_{n+3}\}$ of $\Les^{n+3}$ such that $\{e_1, \ldots, e_{k+1}\}$ is an orthonormal basis of 
$\Les^{k+1}$ with $\<e_1, e_1\>=-1$ and $\{e_{k+2}, \ldots, e_{n+3}\}$ is an orthonormal basis of  $\R^{n-k+2}$, where
$e_2=v_1/\|v_1\|$ and $e_{k+2}=v_2/\|v_2\|$.
Let $a\in\R$ be given by
$$
\frac{\sqrt{a^2+4}-a}{2}=\frac{\|v_1\|}{\|v_2\|}.
$$
Then
$$
\|v_2\|^{-1}v=(\|v_1\|/\|v_2\|)e_2+e_{k+2}=\frac{\sqrt{a^2+4}-a}{2}e_2+e_{k+2}.
$$
Therefore, if we let $\theta\in(0,\pi)$ be a solution of $a=2\cot(\theta)$, and define $\Phi \in O_1(n+3)$ by 
$$\Phi(e_2)=\cos(\theta)e_2+\sin(\theta)e_{k+2},\,\,\,\,\mbox{and}\,\,\,\,\Phi(e_{k+2})=-\sin(\theta)e_2+\cos(\theta)e_{k+2}$$ and $\Phi(e_i)=e_i$ if $1\leq i\leq n+3$, $i\notin\{2, k+2\}$, we are in the situation of Examples \ref{exaweak}-$(i)$, with $M$ being the umbilical hypersurface $M_1$ therein. We conclude that the inclusion $j\colon M\to \Hy^k\times\Sf^{n-k+1}$ is deformable, with a weakly congruent umbilical deformation.\vspace{1ex}\\
$(ii)$  $v_1$ is time-like: 
Fix an orthonormal basis $\{e_1, \ldots, e_{k+1}, e_{k+2}, \ldots, e_{n+3}\}$ of $\Les^{n+3}=\Les^{k+1}\oplus \R^{n-k+2}$ such that $\{e_1, \ldots, e_{k+1}\}$ is an orthonormal basis of 
$\Les^{k+1}$, where
$e_1=v_1/\sqrt{-\<v_1,v_1\>}$ and $e_{k+2}=v_2/\|v_2\|$.
Let $a\in\R$, $a<-2$, be given by
$$
-\frac{a+\sqrt{a^2+4}}{2}=
\frac{\sqrt{-\<v_1,v_1\>}}{\|v_2\|}.
$$
Then
$$
\|v_2\|^{-1}v=\|v_2\|^{-1}
\sqrt{-\<v_1,v_1\>}e_1+e_{k+2}
=-\frac{a+\sqrt{a^2+4}}{2}e_2+e_{k+2}.
$$
Let $\theta\in\R\setminus\{0\}$ be a solution of $a=2\coth(\theta)$, and define $\Phi \in O_1(n+3)$ by 
$$\Phi(e_1)=\cosh(\theta)e_1+\sinh(\theta)e_{k+2},\,\,\,\, \Phi(e_{k+2})=\sinh(\theta)e_2+\cosh(\theta)e_{k+2}$$ and $\Phi(e_i)=e_i$ for $2\leq i\leq n+3$, $i\neq k+2$.
Then we are in the situation of Examples \ref{exaweak}-$(ii)$, and we conclude that the inclusion $j\colon M\to \Hy^k\times\Sf^{n-k+1}$ is deformable, with the umbilical hypersurface $M'=\{v'\}^\perp\cap \Hy^k\times\Sf^{n-k+1}$,  $v'=\frac{a+\sqrt{a^2-4}}{2}e_1+e_{k+2}$, as a weakly congruent deformation.\vspace{1ex}\\
$(iii)$ $v_1$ is light-like:
Let $\{e_1, \ldots, e_{k+1}, e_{k+2}, \ldots, e_{n+3}\}$ be a pseudo-orthonormal basis of $\Les^{n+3}=\Les^{k+1}\oplus \R^{n-k+2}$ such that $\{e_1, \ldots, e_{k+1}\}$ is a pseudo-orthonormal basis of 
$\Les^{k+1}$, with $e_2=v_1$ and $e_{k+2}=v_2/\|v_2\|$. Let $t\in\R$ be given by
$
t=4\|v_2\|^{-1},
$
so that 
$$
\|v_2\|^{-1}v=\frac{t}{4}e_2+e_{k+2}.
$$
Define $\Phi \in O_1(n+3)$ by
$$\Phi(e_1)=e_1-\frac{t^2}{2}e_2+te_{k+2},\,\,\,\,\Phi(e_{k+2})=-te_2+e_{k+2}$$ and $\Phi(e_i)=e_i$ for $2\leq i\leq n+3$, $i\neq k+2$. Then we are in the situation of Examples \ref{exaweak}-$(iii)$. We conclude that the inclusion $j\colon M\to \Hy^k\times\Sf^{n-k+1}$ is deformable, with the umbilical hypersurface $M'=\{v'\}^\perp\cap \Hy^k\times\Sf^{n-k+1}$, 
 $v'=\Phi(v)=-\frac{3t}{4}e_2+e_{k+2}$,  as a weakly congruent deformation.\qed \vspace{1ex}

 In the following example, we combine the isometries $\Phi \in O_1(n+3)$ in parts $(i)$ and $(ii)$ of Examples \ref{exaweak} to construct a two-parameter family of non-umbilical hypersurfaces of $\Hy^{k}\times\Sf^{n-k+1}$ that admit weakly congruent deformations.

\begin{example}\label{nonumbilical}\emph{Fix an orthonormal basis $\{e_1, \ldots, e_{k+1}, e_{k+2}, \ldots, e_{n+3}\}$ of $\Les^{n+3}=\Les^{k+1}\oplus \R^{n-k+2}$ such that $\{e_1, \ldots, e_{k+1}\}$ is an orthonormal basis of 
$\Les^{k+1}$, with $\<e_1, e_1\>=-1$. Define $\Phi \in O_1(n+3)$ by 
$$
\Phi(e_1)=\cosh(\theta)e_1+\sinh(\theta)e_{k+2},\,\,\,\, \Phi(e_{k+2})=\sinh(\theta)e_1+\cosh(\theta)e_{k+2},
$$ 
$$
\Phi(e_2)=\cos(\rho)e_2+\sin(\rho)e_{k+3},\,\,\,\, \Phi(e_{k+3})=-\sin(\rho)e_2+\cos(\rho)e_{k+3}
$$
and $\Phi(e_i)=e_i$ for $i\not \in \{1,2, k+2, k+3\}$. Then $\Phi$ does not preserve the orthogonal decomposition $\Les^{n+3}=\Les^{k+1}\oplus \R^{n-k+2}$.
The subset $M=F_{\Phi}^{-1}(-1,1, -1,1)$ consists of the solutions to the system of equations
$$
\begin{cases}
    -1&=-x_1^2+\sum_{i=2}^{k+1}x_i^2\\
    1&=\sum_{j=k+2}^{n+3}x_j^2\\
    -1&=-(\cosh(\theta)x_1+\sinh(\theta)x_{k+2})^2+(\cos \,\rho x_2-\sin \,\rho x_{k+3})^2+\sum_{i=2}^{k+1}x_i^2\\
    1&=(\sinh(\theta)x_1+\cosh(\theta)x_{k+2})^2+(\sin \,\rho x_2+\cos \, \rho x_{k+3})^2+\sum_{j=k+3}^{n+3}x_j^2.
\end{cases}
$$
Writing $F_\Phi=(F_1, F_2, F_3, F_4)$, a straightforward computation shows that $dF_1-dF_3=dF_4-dF_2$, hence $F_\Phi$ has rank three. It remains to check that $M\neq  \emptyset$. Taking the differences between the first and the third, and then between the second and fourth equations, shows that the preceding system of equations is equivalent to
$$
\left\{ \begin{array}{l} -x_1^2+\sum_{i=2}^{k+1} x_i^2=-1\\
\sum_{j=k+2}^{n+3} x_j^2=1\\
\sinh^2\theta x_1^2+\sinh^2\theta x_{k+2}^2+2\cosh \theta\sinh \theta x_1x_{k+2}\\
+\sin^2\rho x_2^2-\sin^2\rho x_{k+3}^2+2\cos \rho\sin \rho x_2x_{k+3}=0.
\end{array}
\right.
$$
Note that the last equation only contains quadratic terms, hence it represents a cone through the origin. Note also that any generator of the light cone intersects $\mathbb{H}^k\times \mathbb{S}^{n-k+1}$. Hence, it suffices to show that the above system has some solution in the light cone. We look for solutions $(x_1, \ldots, x_{k+3})$ such that 
$x_1=1$,  $x_{k+3}=x_2$ and $x_i=0$ for $i\not\in \{1, 2, k+3\}$, and hence  $2x^2_2=1-x_{k+2}^2$. In this case, the last equation reduces to
$$
\sinh^2\theta +\sinh^2\theta x_{k+2}^2+2\cosh \theta\sinh \theta x_{k+2}\
+\cos \rho\sin \rho (1-x_{k+2})=0,
$$
that is,
$$
(\sinh^2\theta -\cos \rho \sin \rho) x_{k+2}^2 + 2 \cosh \theta \sinh \theta x_{k+2} +(\sinh^2\theta+\cos \rho \sin\rho)=0,$$
whose discriminant is
\begin{eqnarray*}
D&=&4\cosh^2 \theta \sinh^2 \theta - 4(\sinh^2 \theta-\cos \rho \sin\rho)(\sinh^2 \theta+\cos \rho \sin\rho)\\
&=&4(\cosh^2 \theta \sinh^2 \theta-(\sinh^4 \theta-\cos^2 \rho \sin^2\rho))\\
&=&4(\sinh^2 \theta+\cos^2 \rho \sin^2\rho)\geq 0.
\end{eqnarray*}}
\end{example}

\section{Truly deformable hypersurfaces}\label{Sec:truly}

In this section, we study the truly deformable hypersurfaces of $\Hy^{k}\times\Sf^{n-k+1}$, $2\leq k\leq n-1$.

\subsection{Necessary conditions} 

  We start by showing that it is enough to consider hypersurfaces whose index of relative nullity is less than or equal to $n-2$ at every point. \vspace{1ex}

  \begin{proposition}\label{split2}
If the hypersurface $f\colon M^n\to \Hy^{k}\times\Sf^{n-k+1}$, $2\leq k\leq n-1$, is either totally geodesic or has a constant index of relative nullity $\nu_f=n-1$ at every point, then $f$ splits locally.  
\end{proposition}

\proof
Let $\Delta$ stand for the relative nullity distribution of $f$. If $\Delta=TM$, that is, $f$ is totally geodesic, applying  
\eqref{codazzi} to linearly independent vector fields $X$ and $Y$ implies that the vector field $\xi_f$, and hence the tensor $S_f$, vanishes identically.
Thus the statement follows from  Proposition \ref{splits}. This also follows from the classification of totally geodesic submanifolds of $\mathbb{H}^k\times \mathbb{S}^{n-k+1}$ (see, e.g., Proposition $4$ of \cite{JT}).

Assume now that $\Delta$ has dimension $n-1$ everywhere. Applying  again
\eqref{codazzi} to linearly independent vector fields $X, Y\in \Gamma(\Delta)$ implies that $\Delta\subset\ker S_f$. 
Now let $X\in\Gamma(\Delta^\perp)$, $Y\in\Gamma(\Delta)$, and denote by $\lambda_f$ the non-zero principal curvature of $A_f$. Then, applying \eqref{codazzi} for $X$ and $Y$ gives
$$
S_f(X)=\lambda_f\<\nabla_YY,X\>.
$$
On the other hand, evaluating \eqref{derS} in $Y$ yields
$$
Y(S_f(Y))-S_f(\nabla_YY)=-S_f(X)\<\nabla_YY,X\>=0.
$$
By the previous equations, we necessarily have $S_f(X)=0$, and hence, by Proposition \ref{splits}, the hypersurface $f$ splits locally.
\qed
\vspace{1ex}

\begin{proposition}\label{commoneig}
Let $f\colon M^n\to \Hy^{k}\times\Sf^{n-k+1}$, $n\geq 5$, $2\leq k\leq n-1$, be an isometric immersion whose index of relative nullity is less than or equal to $n-2$ at every point. Assume that $f$ is truly deformable and let $g\colon M^n\to \Hy^{k}\times\Sf^{n-k+1}$ be a true deformation of $f$.  
Then $A_f$, $A_g$ and $R_f-R_g$ share a common eigenbundle of rank at least $(n-2)$ on any connected component of an open and dense subset of $M^n$. 
\end{proposition}
\proof
It follows from the assumptions and Proposition \ref{prop:null}  that the associated flat bilinear form $\beta$ is not null on any open subset of $M^n$. Then, we consider a connected component of the open and dense subset where $\beta$ is not null. Notice that $\beta(x)$ has trivial kernel for any $x\in M^n$, for if $T\in\mathcal{N}(\beta(x))$, that is, $\beta(X,T)=0$ for all $X\in T_xM$, then  
$$
\<\psi_+X,T\>=0=\<\psi_-X,T\>,
$$
which implies that $\<X,T\>=0$  for all  $X\in T_xM$, and hence $T=0$.
We are therefore in a position to apply Lemma 4.22 in \cite{DT}, which states that, at each point $x\in M^n$, there exists an orthogonal decomposition
$$
\R^{2,2}=W_1^{1,1}\oplus W_2^{1,1}
$$
such that $\beta(x)$ decomposes accordingly as $\beta(x)=\beta_1(x)+\beta_2(x)$, where $\beta_1(x)$ is nonzero and null, $\beta_2(x)$ is flat and $\dim\mathcal{N}(\beta_2(x))\geq n-2$.

From now on, we argue on a connected component of an open and dense subset of $M^n$ where $\beta_2$ is not null and $\dim\mathcal{N}(\beta_2)\geq n-2$ is constant. 
For any $T\in\Gamma(\mathcal{N}(\beta_2))$, we have
$$
\<\!\<\,\beta(X,T),\beta(Y,Z)\>\!\>=0
$$
for all $X,Y,Z\in\mathfrak{X}(M)$. The previous equation is equivalent to
\be\label{reltens}
\<A_fX,T\>A_f-\<A_gX,T\>A_g+\<X,T\>(R_f-R_g)+\<(R_f-R_g)X,T\>I=0
\ee
for all $X\in\mathfrak{X}(M)$.
If $T$ is not a principal direction of $g$, then there exists $X$ orthogonal to $T$ such that $\<A_gX,T\>\neq 0$. For such $X$ and $T$, Eq. \eqref{reltens} becomes
$$
\<A_fX,T\>A_f-\<A_gX,T\>A_g+\<(R_f-R_g)X,T\>I=0.
$$
Notice that we necessarily have $\<A_fX, T\>\neq0$, because, otherwise, $A_g$ would be a multiple of the identity endomorphism, in contradiction with the fact that  $\<A_gX,T\>\neq 0$ with  $X$ orthogonal to $T$. This implies that $A_f$ and $A_g$ commute, and, therefore,  they also commute with $(R_f-R_g)$ by \eqref{reltens}. Hence they are simultaneously diagonalizable.
From the assumption that $T$ is not an eigenvector of $A_g$, it follows that there exist two principal directions, $X_1$ and $X_2$, associated with distinct eigenvalues $\mu_1\neq\mu_2$ of $A_g$, such that $\<X_i,T\>\neq 0$, $i=1,2$. Denote by $\lambda_i$ and $r_i$ the eigenvalues of $A_f$ and $(R_f-R_g)$, respectively, that correspond to $X_i$. Then 
$$
\lambda_iA_f-\mu_iA_g+(R_f-R_g)+r_iI=0
$$
for $i=1,2$. Taking the difference, we obtain
$$
(\lambda_1-\lambda_2)A_f-(\mu_1-\mu_2)A_g+(r_1-r_2)I=0,
$$
and we conclude that $\lambda_1\neq\lambda_2$. Then, evaluating the preceding equation on $X_i$, $i=1,2$, and taking the difference, yield
$$
(\lambda_1-\lambda_2)^2=(\mu_1-\mu_2)^2.
$$
Changing the orientation of $g$, if necessary, we can assume that $(\lambda_1-\lambda_2)=(\mu_1-\mu_2)$. Denoting $\rho=(r_1-r_2)/(\mu_1-\mu_2)$, we see that \eqref{cong1} holds.
It follows that $\mu_i=\lambda_i+\rho$,  hence $2r_i=2\rho\lambda_i+\rho^2$.
Substituting \eqref{cong1} and the previous relations in \eqref{reltens} yields
\begin{align*}
    0&=\lambda_iA_f-(\lambda_i+\rho)(A_f+\rho I)+(R_f-R_g)+r_i I\\
    &=R_f-R_g-\rho A_f-\frac{\rho^2}{2}I,
\end{align*}
which is \eqref{cong2}. Thus, Proposition \ref{congL} and Lemma \eqref{12imp3} lead to a contradiction.

Therefore, any $T\in\Gamma(\mathcal{N}(\beta_2))$ is an eigenvector of $A_g$, and \eqref{reltens} implies that $\mathcal{N}(\beta_2)$ is in fact a common eigenbundle of $A_f$, $A_g$ and $(R_f-R_g)$.\vspace{1ex}
\qed

Proposition \ref{commoneig} implies the following result in the line of the Beez-Killing and Cartan's theorems (see Theorem $4.13$ and Corollary $9.25$ in \cite{DT}).

\begin{corollary} \label{beez} If a hypersurface $f\colon M^n\to \mathbb{H}^k\times \mathbb{S}^{n-k+1}$, $n\geq 5$, $2\leq k\leq n-1$, has no principal curvature of multiplicity greater than $n-3$, then it is weakly rigid.
\end{corollary}

   For $n\geq 6$, Proposition \ref{commoneig}  can be improved as follows. 
   
  \begin{proposition} \label{commoneig2}
    Let $f\colon M^n\to \Hy^{k}\times\Sf^{n-k+1}$, $n\geq 6$, $2\leq k\leq n-1$, be an isometric immersion whose index of relative nullity is less than or equal to $n-2$ at every point. Assume that $f$ is truly deformable and let $g\colon M^n\to \Hy^{k}\times\Sf^{n-k+1}$ be a true deformation of $f$.  Then $A_f$ and $A_g$  share a common eigenbundle $\Delta$ of rank $n-1$ or $n-2$ on each connected component of an open and dense subset of $M^n$, the eigenvalues of $A_f|_\Delta$ and $A_g|_\Delta$ coincide (up to sign), $R_f|_{\Delta}=R_g|_{\Delta}$, and
    \be\label{reltens3}
\lambda(A_f-A_g)+(R_f-R_g)=0,
\ee
where $\lambda$ is the common eigenvalue of $A_f$ and $A_g$ on $\Delta$.
\end{proposition}
\proof
Let $\Delta$ be the eigenbundle given by Proposition \ref{commoneig} on a connected component $U$ of an open and dense subset of $M^n$, and let $\lambda_f$, $\lambda_g$ and $r$ the corresponding eigenvalues of $A_f|_{\Delta}$, $A_g|_{\Delta}$ and $(R_f-R_g)|_{\Delta}$ on $U$, respectively. Then \eqref{reltens} becomes 
\be\label{reltens2}
\lambda_fA_f-\lambda_gA_g+(R_f-R_g)+rI=0.
\ee

Assume first that $k\leq (n+1)/2$.  Since 
$\dim\ker (I-R_f)\geq n-k$ by \eqref{dimkerR}, the assumption on $n$ implies that $\dim\ker (I-R_f)\geq n-k\geq (n-1)/2$. Hence 
$$
\dim(\Delta\cap\ker (I-R_f))\geq n-2+\frac{n-1}{2}-n=\frac{n-5}{2}>0.
$$
Let $X_1\in \Gamma(\Delta\cap\ker(I-R_f))$. Then
$$
(R_f-R_g)X_1=R_fX_1-R_gX_1=X_1-R_gX_1,
$$
and thus
$$
R_gX_1=(1-r)X_1.
$$
Hence $X_1$ is also an eigenvector of $R_g$, and we have three possibilities for the value of $(1-r)$, namely,  $1-r\in\{0,1,r_g\}$, where $0<r_g<1$. Following the same arguments for $g$, if $X_2\in \Gamma(\Delta\cap\ker(I-R_g))$, in a similar way we see that
$$
R_fX_2=(1+r)X_2,
$$
and that $1+r\in\{0,1,r_f\}$, where $0<r_f<1$.
If $r=1$, then $1+r=2$, but this is not possible, since $0<r_f<1$. If $1-r=r_g$, then $r=1-r_g>0$, and hence $1+r=r_f$, which is not possible either. 
Thus, the only possibility is that $r=0$.

Similarly, if $k\geq (n+1)/2$, from $\dim\ker R_f\geq k-1$ and the assumption on $n$, we see that $\dim(\Delta\cap\ker R_f)\geq\frac{n-5}{2}>0$. Taking $X_1\in\Delta\cap\ker R_f$ and $X_2\in\Delta\cap\ker R_g$, and arguing as in the preceding case, we obtain $R_gX_1=-rX_1$ and $R_fX_2=rX_2$. Recall that the eigenvalues of $R_f$ and $R_g$ are all nonnegative, hence the only possibility is again that $r=0$.
Thus $(R_f-R_g)|_\Delta=0$ in any case, which, together with \eqref{reltens2}, gives $\lambda_f^2-\lambda_g^2=0$.
Replacing $N_g$ by $-N_g$,  if necessary, we can assume that 
$
\lambda_f=\lambda_g.
$ 
Thus Eq. \eqref{reltens2} becomes \eqref{reltens3}.

It remains to argue that  $\dim \Delta$ can not be equal to $n$ on any connected component $U$. 
Indeed, if otherwise, since $r=0$, then $R_f=R_g$, and from $\lambda_f=\lambda_g$ we obtain $A_f=A_g$. 
Moreover, it follows from Lemma~\ref{12imp3} that $S_f=S_g$. Thus $f|_U$ and $g|_U$  would be congruent by Proposition \ref{cong}, 
a contradiction.
\qed
\vspace{2ex}

We conclude this section with two facts that will be used in the case-by-case study carried out in the next section.

\begin{proposition}\label{umbeig}
    Let $f\colon M^n\to\Hy^k\times\Sf^{n-k+1}$ be an isometric immersion and let $N\in\Gamma(N_fM)$ be a unit normal vector field with associated shape operator $A$.
    If $\lambda\in C^\infty(M)$ is a principal curvature of $f$, then the associated eigenbundle $\Delta_\lambda=\ker(A-\lambda I)$ is an umbilical distribution.
\end{proposition}
\proof
We use the conformal diffeomorphism $\Psi\colon\Hy^{k}\times\Sf^{n-k+1}\to\R^{n+1}\setminus\R^{k-1}$. 
The vector field $\Psi_*N$ is normal to the conformal immersion $\hat{f}=\Psi\circ f$. 
The relation between the Levi-Civita connections associated with the (conformal) metrics of $\Hy^{k}\times\Sf^{n-k+1}$ and $\R^{n+1}\setminus\R^{k-1}$ implies that
the shape operator $A_{\Psi_*N}$ of $\hat f$ satisfies $A_{\Psi_*N}=\rho_1 A+\rho_2 I$ for some smooth functions $\rho_1$ and $\rho_2$. 
Therefore $\Delta_\lambda$ is the eigenbundle associated with a principal curvature of $\hat{f}$. 
Thus $\Delta_\lambda$ is an umbilical distribution with respect to the metric induced by $\hat{f}$. Since umbilicity of distributions is invariant under conformal changes of the metric, then $\Delta_\lambda$ is also  umbilical with respect to the metric induced by~$f$.
\qed

\begin{lemma}\label{lemdeltaumb}
Let $f\colon M^n\to \Hy^{k}\times\Sf^{n-k+1}$, $n\geq 6$, $2\leq k\leq n-1$, be an isometric immersion whose index of relative nullity is less than or equal to $n-2$ at every point. Assume that $f$ is truly deformable and let $g\colon M^n\to \Hy^{k}\times\Sf^{n-k+1}$ be a true deformation of $f$. Let $\Delta$ be the common eigenbundle of $A_f$, $A_g$ and $R_f-R_g$ given by Proposition \ref{commoneig2}. If $\lambda\neq 0$, where $A_f|_{\Delta}=A_g|_{\Delta}=\lambda I$, then $\Delta'=\ker(A-\lambda I)\cap\ker(A-\lambda I)$ is an umbilical distribution such that $\Delta\subset\Delta'$ and $\Delta'\subset\ker(R_f-R_g)$.

\end{lemma}
\proof
From Proposition \ref{umbeig}, as the intersection of umbilical distributions, $\Delta'$ is also umbilical. 
Clearly, $\Delta\subset\ker(A_f-\lambda I)\cap\ker(A_g-\lambda I)$. 
Finally, given that $\lambda\neq 0$, we see from \eqref{reltens3} that $\Delta'\subset\ker(R_f-R_g)$.  \vspace{1ex}
\qed

\subsection{A case-by-case study}

We now proceed to describe the truly deformable hypersurfaces $f\colon M^n\to \Hy^{k}\times\Sf^{n-k+1}$, $n\geq 6$, $2\leq k\leq n-1$, that do not split in any open subset. Let $g\colon M^n\to \Hy^{k}\times\Sf^{n-k+1}$ be a true deformation of $f$, and let $\Delta$ and $\lambda$ be the common eigenbundle and the common eigenvalue of $A_f|_\Delta$ and $A_g|_\Delta$, respectively, given by Proposition~\ref{commoneig2}. 
We study the possible cases separately according to the rank of $\Delta$ and depending on whether $\lambda$ is identically zero or vanishes nowhere. From now on, we restrict the immersion to a connected component of an open and dense subset where the assumptions hold, without explicitly mentioning it. 

\vspace{2ex}

\subsubsection[Case lambda=0 and dimDelta=n-2]{Case $\lambda=0$ and $\dim \Delta=n-2$}

The next result deals with the case in which 
$\lambda=0$ and $\dim \Delta=n-2$.


\begin{proposition}\label{ruled}
Let $f\colon M^n\to \Hy^k\times \Sf^{n-k+1}$, $n\geq 6$, $2\leq k\leq n-1$, be a truly deformable hypersurface that does not split in any open subset, and let $g\colon M^n\to \Hy^{k}\times\Sf^{n-k+1}$ be a true deformation of $f$ such that  the common eigenbundle $\Delta$ and and the common eigenvalue $\lambda$ of $A_f|_\Delta$ and $A_g|_\Delta$, respectively, given by Proposition~\ref{commoneig2}, satisfy $\lambda=0$ and $\dim \Delta=n-2$. Then $\Delta$ coincides with the relative nullity distribution of $f$ and $f$ is ruled.
\end{proposition}
\proof
By \eqref{reltens3}, we have $R_f=R_g$, hence either $S_f=0=S_g$ or both one-forms are nonzero.
In the former case,  both $f$ and $g$ split locally by Proposition \ref{splits},  
which is ruled out by our assumptions.
Thus,  $S_f$ and $S_g$ are both nonzero and \eqref{eq:pi1} implies that
$
S_f=\pm S_g.
$
After changing $N_g$ by $-N_g$, if necessary, we may assume that $S_f=S_g$.
Notice that \eqref{codazzi} and the fact that $n-2\geq 2$ imply that $\Delta\subset \ker S_f=\ker S_g$. 
Moreover, since $f$ and $g$ do not split in any open subset, it follows from Proposition \ref{split2} that $\Delta=\ker A_f=\ker A_g$, 
hence $\Delta$ is an umbilical distribution.
Applying \eqref{codR} for $f$ and $g$ gives
$$
\<\xi_f, Y\>A_fX-\<\xi_f,X\>A_fY=\<\xi_g, Y\>A_gX-\<\xi_g, X\>A_gY
$$
for all $X,Y\in\mathfrak{X}(M)$. Taking $X\in\Gamma((\ker S_f)^\perp)$ and $Y\in\Gamma(\ker S_f)$, we obtain
$$
\<\xi_f,X\>A_fY=\<\xi_g, X\>A_gY,
$$
which implies that $A_f$ and $A_g$ coincide on $\ker S_f=\ker S_g$.
By Proposition \ref{cong}, we must have  $A_fX\neq A_gX$,  otherwise, $f$ and $g$ would be congruent. 

Since $R_f=R_g$, the Gauss equation \eqref{Gauss} for $f$ and $g$ gives
$$
(A_fX\wedge A_fY)Z=(A_gX\wedge A_gY)Z
$$
for all $X,Y,Z\in\mathfrak{X}(M)$.
Take a (local) orthonormal frame $\{X, Y\}$ of $\Delta^\perp$ with $X\in\Gamma((\ker S_f)^\perp)$ and $Y\in\Gamma(\ker S_f\cap\Delta^\perp)$. 
Since $A_fY=A_gY$, then
$$
\<A_fY,Z\>(A_fX-A_gX)=\<(A_f-A_g)X,Z\>A_fY
$$
for any $Z\in\mathfrak{X}(M)$. Taking $Z=Y$ yields
$$
\<A_fY,Y\>(A_f-A_g)X=0,
$$
and the fact that $f$ and $g$ are not congruent implies that $\<A_fY,Y\>=0$.
Therefore, the shape operators of $f$ and $g$ on $\Delta^\perp$, with respect to $X$ and $Y$, have the form
$$
A_f=\begin{bmatrix}
\lambda_f & \mu\\
\mu & 0
\end{bmatrix}
\;\;\mbox{and}\;\;
A_g=\begin{bmatrix}
\lambda_g & \mu\\
\mu & 0
\end{bmatrix}
$$
Given $T\in\Gamma(\Delta)$ of unit length,  the $Y$ component of \eqref{codazzi}, evaluated in 
$X$ and $T$, gives
\begin{align*}
0&=\<(\nabla_XA_f)T-(\nabla_TA_f)X,Y\>\\
&=-\mu\<\nabla_XT,X\>-\<\nabla_T(\lambda_fX+\mu Y),Y\>+\mu\<\nabla_TX,X\>\\
&=\mu\<T,\nabla_XX\>-\lambda_f\<\nabla_TX,Y\>-T(\mu).
\end{align*}
Similarly, 
$$
\mu\<T,\nabla_XX\>-\lambda_g\<\nabla_TX,Y\>-T(\mu)=0,
$$
which implies that $\<\nabla_TX,Y\>=0$.
As for the $T$ component of the same equation, we have
\begin{align*}
    S_f(X)&=\<(\nabla_XA_f)T-(\nabla_TA_f)X,T\>\\
    &=-\<\nabla_T(\lambda_fX+\mu Y),T\>\\
    &=-\lambda_f\<\nabla_TX,T\>-\mu\<\nabla_TY,T\>,
\end{align*}
and, similarly,
$$
S_g(X)=-\lambda_g\<\nabla_TX,T\>-\mu\<\nabla_TY,T\>.
$$
Hence $\<\nabla_TX,T\>=0$.
Let us now take the $Y$ component of \eqref{codazzi} evaluated in $X$ and $Y$, that is,
\begin{align*}
    S_f(X)&=\<(\nabla_XA_f)Y-(\nabla_YA_f)X,Y\>\\
    &=2\mu\<\nabla_XX,Y\>-\lambda_f\<\nabla_YX,Y\>-Y(\mu).
\end{align*}
Similarly,
$$
S_g(X)=2\mu\<\nabla_XX,Y\>-\lambda_g\<\nabla_YX,Y\>-Y(\mu).
$$
Thus, we necessarily have $\<\nabla_YX,Y\>=0$.
In an analogous way, the $T$ component of \eqref{codazzi} evaluated in $X$ and $Y$, for $f$ and $g$, gives
$\<\nabla_YX,T\>=0$.
Summarizing, we have shown that 
\[\<\nabla_TX,Y\>=\<\nabla_TX,T\>=\<\nabla_YX,Y\>=\<\nabla_YX,T\>=0,\]
which, together with the fact that $\Delta$ is an umbilical distribution,  implies that $\ker S_f=\ker S_g=\Delta\oplus \mbox{span}\,\{Y\}$ is a totally geodesic distribution. Moreover, since $A_f|_\Delta=0=\<A_fY, Y\>$, the restriction of $f$ to each of its leaves is totally geodesic. Thus $f$ is ruled.
\qed
\subsubsection[Case lambda not 0 and dimDelta=n-2]{Case $\lambda\neq 0$ and $\dim \Delta=n-2$}
Let $f\colon M^n\to \Hy^k\times \Sf^{n-k+1}$, $n\geq 6$, $2\leq k\leq n-1$, be a truly deformable hypersurface, and
let  $g\colon M^n\to \Hy^k\times \Sf^{n-k+1}$ be a true deformation of $f$ such that the common eigenbundle $\Delta$ of $f$ and $g$,
and their common principal curvature $\lambda$ on $\Delta$, given by Proposition \ref{commoneig2}, satisfy $\lambda\neq 0$ and $\dim \Delta=n-2$. By Lemma \ref{lemdeltaumb}, we may take $\Delta$ as $\Delta=\ker(A_f-\lambda I)\cap\ker(A_g-\lambda I)$, which is an umbilical distribution.
In this case, it follows from \eqref{codazzi} for $f$ and $g$ that $S_f(T)=T(\lambda)=S_g(T)$ for any $T\in\Gamma(\Delta)$. Let $\{X,Y\}$  be an orthonormal frame of $\Delta^\perp$ given by eigenvectors of $A_f-A_g$, and such that 
$$(A_f-A_g)X=\lambda_1 X\,\,\,\,\mbox{and}\,\,\,\,(A_f-A_g)Y=\lambda_2 Y.$$  
It follows from \ref{reltens3}  that $X$ and $Y$ are also eigenvectors of $(R_f-R_g)$ with corresponding eigenvalues 
\be\label{rlambda}
r_i=-\lambda\lambda_i, \quad 1\leq i\leq 2.
\ee
The difference of \eqref{derR} for $f$ and $g$, evaluated in $X$ and $T\in\Gamma(\Delta)$, gives
$$
(\nabla_X(R_f-R_g))T=-(R_f-R_g)\nabla_XT=S_f(T)A_fX-S_g(T)A_gX=\lambda_1S_f(T)X.
$$
Then, the $X$ and $Y$ components give, respectively,
\be\label{XXT}
r_1\<\nabla_XX,T\>=\lambda_1S_f(T)
\ee
and 
\be\label{XTY}
r_2\<\nabla_XT,Y\>=0.
\ee
Exchanging the places of $X$ and $T$ in \eqref{derR}, we obtain
$$
T(r_1)X+r_1\nabla_TX-(R_f-R_g)\nabla_TX=\lambda(S_f-S_g)(X)T.
$$
The $X$, $Y$ and $T$ components of the equation above give, respectively,
\be\label{Tr1}
T(r_1)=0,
\ee
\be\label{TXY}
(r_1-r_2)\<\nabla_TX,Y\>=0
\ee
and 
\be\label{deltaX}
r_1\<\nabla_TX,T\>=\lambda(S_f-S_g)(X).
\ee
Similarly, using $Y$ and $T$ we have
\begin{align}
r_2\<\nabla_YY,T\>&=\lambda_2S_f(T),\label{YYT} \\
r_1\<\nabla_YT,X\>&=0,\label{YTX}\\
T(r_2)&=0,\label{Tr2}\\
(r_2-r_1)\<\nabla_TY,X\>&=0 \nonumber
\end{align}
and
\be\label{deltaY}
r_2\<\nabla_TY,T\>=\lambda(S_f-S_g)(Y).
\ee
Since $\lambda\neq 0$, then \eqref{rlambda}, \eqref{XXT} and \eqref{Tr1} yield
\be\label{Tlambda1}
T(\lambda_1)+\lambda_1\<\nabla_XT,X\>=0,
\ee
and, in a similar manner, we obtain 
\be\label{Tlambda2}
T(\lambda_2)+\lambda_2\<\nabla_YT,Y\>=0.
\ee\vspace{1ex}
The next proposition investigates the case in which $S_f=S_g$.

\begin{proposition}\label{sameS}
Let $f\colon M^n\to \Hy^k\times \Sf^{n-k+1}$, $n\geq 6$, $2\leq k\leq n-1$, be a truly deformable hypersurface that does not split, and let $g\colon M^n\to \Hy^k\times \Sf^{n-k+1}$ be a true deformation of $f$.  Assume that the common eigenbundle $\Delta$ of $f$ and $g$ and the common principal curvature $\lambda$ of $A_f|_{\Delta}$ and $A_g|_{\Delta}$, given by Proposition \ref{commoneig2},  satisfy  $\lambda\neq 0$ and $\dim \Delta=n-2$. If, in addition,  $S_f=S_g$, then the subset where $\Delta\subset\ker S_f$ has empty interior and, in the open and dense subset  $\mathcal{V}$ where $\Delta\not \subset\ker S$, $f|_\mathcal{V}$ and $g|_\mathcal{V}$ are locally either doubly rotational hypersurfaces, or, if $k=2$ (respectively, $k=n-1$),  rotational hypersurfaces of type $(I)$ (respectively, type $(II)$),  whose profiles are, in any case, isometric three-dimensional hypersurfaces $f_0, g_0\colon M_0^3\to \Hy^2\times\Sf^2$. 
\end{proposition}
\proof
Since $S_f=S_g$, then $\xi_f=\xi_g$ and $R_f$ and $R_g$ share the same eigenvector $\xi_f$. 
Taking $T\in\Gamma(\Delta)$ of unit norm, the difference of \eqref{derS} for $f$ and $g$ evaluated in $T$ gives
$$
\lambda(t_g-t_f)=0.
$$
Since $r_f=1-t_f$ and $r_g=1-t_g$,  it follows from the preceding equation and the assumption that $\lambda \neq 0$ that $r_f=r_g$, hence $\xi_f=\xi_g\in\Gamma(\ker (R_f-R_g))$.

First assume that $\xi_f\in\Gamma(\Delta^\perp)$, or equivalently, that $\Delta\subset\ker S_f$ on some open subset $U\subset M^n$.
In this case, Eq. \eqref{derS}, evaluated in $T\in\Gamma(\Delta)$, gives
$$
-S_f(\nabla_TT)=\lambda(t_f-\<R_fT,T\>).
$$
Since $\Delta$ is an umbilical distribution,  the left-hand side of the previous equation does not depend on the choice of $T$, hence, the same holds for the right-hand side. Observe that the assumption $n\geq 6$ implies that the common eigenbundle $\Delta$ of $f$ and $g$ intersects $\ker(I-R_f)$ or $\ker R_f$. This and the previous equation show that $\Delta$ must be contained in one of these kernels, otherwise  $\<R_fT,T\>$ would depend on $T$. 
This argument also holds for  $g$, thus $R_fT=R_gT=0$ or $R_fT=R_gT=T$ for any $T\in\Gamma(\Delta)$. 
Since $S_f\neq0$, then  $\dim\ker R_f=\dim\ker R_g$ and $\dim\ker (I-R_f)=\dim\ker(I-R_g)$ by \ref{dimkerR}. Given that $\xi_f\in\Gamma(\Delta^\perp)$ is a common eigenvector of $R_f$ and $R_g$ corresponding to the same eigenvalue,  we necessarily have $R_f=R_g$,  hence $A_f=A_g$, which is a contradiction by Proposition \ref{cong}.

Thus  $\xi_f\not\in\Gamma(\Delta^\perp)$, or equivalently, $\Delta\not\subset \ker S_f$, on an open and dense subset. 
It follows from \eqref{deltaX} and \eqref{deltaY} that
$$
r_1\<\nabla_TX,T\>=0=r_2\<\nabla_TYT\>
$$
for any $T\in\Gamma(\Delta)$.
Observe that $r_1$ and $r_2$ cannot be both equal to $0$, otherwise $f$ and $g$ would be congruent by Proposition \ref{cong}.
Since $S_f=S_g$, then \eqref{derS2} for $f$ and $g$ gives 
\be\label{eq:RA1}
t_f(A_f-A_g)=R_fA_f-R_gA_g.
\ee
Suppose that $r_2=0$, and hence that $\lambda_2=0$. 
The previous equation, evaluated in $Y$, gives
$$
(R_f-R_g)A_fY=0.
$$
Then $\<A_fY, X\>=0$, hence $Y$ is a common eigenvector of $A_f$ and $A_g$, with respect to the same eigenvalue. 
Consequently, $A_fX=\mu X$ and $A_gX=\nu X$, where $\lambda_1=\mu-\nu$.
Denote $L=\Delta\oplus\spa\{Y\}$, and notice that $(A_f-A_g)|_L=(R_f-R_g)|_L=0$. 
Choosing $Z\in\Gamma(L)$, then the $Z$ component of \eqref{eq:RA1}, evaluated in $X$, gives
$$
\lambda_1\<X,R_fZ\>=0.
$$
This implies that $X$ is an eigenvector of $R_f$ and $R_g$. 
Since $\Delta\not\subset\ker S_f$, then $S_f(X)=0$ with $R_fX=X$ or $R_fX=0$, and similarly for $R_g$.
The subbundle $L$ is invariant by $R_f$ and $R_g$, and $R_f|_L=R_g|_L$. 
Since $\xi_f=\xi_g\in\Gamma(L)$, then $\bar{\Delta}=\ker S_f\cap L$ is invariant by $R_f$ and $R_g$, with $R_f|_{\bar{\Delta}}=R_g|_{\bar{\Delta}}$.
The eigenvalues of $R_f$ and $R_g$ corresponding to $X$ are different, for $r_1\neq 0$.
Using \ref{dimkerR}, we see that the multiplicities of the eigenvalues of $R_f$ and $R_g$ in $\bar{\Delta}$ do not coincide, and this is a contradiction.

Thus, we necessarily have $r_1r_2\neq 0$. Since $\xi_f\in\Gamma(\ker (R_f-R_g))$, then $\xi_f\in\Gamma(\Delta)$. 
It follows from \eqref{deltaX} and $\eqref{deltaY}$ that $\Delta$ is a totally geodesic distribution.
Given $T\in\Gamma(\Delta)$ of unit norm, then the $T$ component of $\eqref{eq:RA1}$, evaluated in $X$ and $Y$, gives 
$$
\lambda_1\<X,R_fT\>=0=\lambda_2\<Y,R_fT\>.
$$
This shows that $R_f$ leaves $\Delta$ and $\Delta^\perp$ invariant, and the same holds for $R_g$.
From \eqref{dimkerR}, we know that the multiplicities of the eigenvalues $0$ and $1$ for $R_f$ and $R_g$ coincide, since they depend only on the fixed number $k$.
Since $R_f|_{\Delta}=R_g|_{\Delta}$, then $R_f|_{\Delta^\perp}$ and $R_g|_{\Delta^\perp}$ have the same eigenvalues.
If $R_f|_{\Delta^\perp}=I$, then $R_g|_{\Delta^\perp}=I$, hence the immersions are congruent by \eqref{reltens3} and Proposition \ref{cong}, a contradiction.
The same holds if $R_f|_{\Delta^\perp}=0$. 
Thus both $0$ and $1$ are eigenvalues of $R_f|_{\Delta^\perp}$ and $R_g|_{\Delta^\perp}$.

Applying \eqref{derR} for $\bar{\xi}=\xi_f/\|\xi_f\|$ gives
\begin{align*}
\nabla_{\bar{\xi}}R_f\bar{\xi}-R_f\nabla_{\bar{\xi}}\bar{\xi}&=\|\xi_f\|\lambda\bar{\xi}+\lambda\xi_f\\
\bar{\xi}(r_f)\bar{\xi}-(R_f-r_fI)\nabla_{\bar{\xi}}\bar{\xi}&=2\lambda\xi_f.
\end{align*}
Notice that $\nabla_{\bar{\xi}}\bar{\xi}$ is orthogonal to $\bar{\xi}$. 
Since $r_f$ is an eigenvalue of multiplicity one, the previous equation implies that
\be\label{xi1}
\nabla_{\bar{\xi}}\bar{\xi}=0.
\ee
Let $\Delta_0$ be given by $\Delta_0=\Delta^\perp\oplus\spa\{\bar{\xi}\}$.
Taking any  $Z\in\Gamma(\Delta^\perp)$, we have from \eqref{derS2} that $\nab_Z\bar{\xi}\in\Gamma(\Delta_0)$. 
Moreover, the fact that $\Delta$ is totally geodesic implies that $\nabla_{\bar{\xi}}Z\in\Gamma(\Delta_0)$.
This together with \eqref{xi1}, \eqref{XXT}, \eqref{XTY}, \eqref{YYT} and \eqref{YTX} show that $\Delta_0$ is totally geodesic.

Let $\Delta_i$, $i=1,2$, be given by $\Delta_1=\ker(R_f)\cap\Delta$ and $\Delta_2=\ker(I-R_f)\cap\Delta$.  Note that $\dim \Delta_1\geq k-3$ and $\dim \Delta_2\geq n-k-2$. Thus $\Delta_1$ (respectively, $\Delta_2) $ might be trivial only if $k\in \{2,3\}$ (respectively, $k\in \{n-1, n-2\}$). 
If $\Delta_2$ is nontrivial, taking $T\in\Gamma(\Delta_2)$,  Eq. \eqref{derS2} gives
$$
\nabla_T\xi_f=T(\|\xi_f\|)\bar{\xi}+\|\xi_f\|\nabla_T\bar{\xi}=\lambda(t_f-1)T.
$$
Therefore $\|\xi_f\|$ is constant along $\Delta_2$, and $\Delta_2$ is an umbilical distribution. 
Since $S_f(T)=T(\lambda)$ for $T\in\Gamma(\Delta)$, then $\lambda$ is constant along $\Delta_2$. Moreover, \eqref{dert} implies that $t_f$ is also constant along $\Delta_2$.
Hence, the mean curvature vector field $\delta_2=\lambda(1-t_f)/\|\xi_f\|\bar{\xi}$ of  $\Delta_2$ satisfies
$$
(\nabla_T\delta_2)_{\Delta_0}=0.
$$
This shows that $\Delta_2$ is in fact a spherical distribution.
A similar argument shows that also $\Delta_1$, if nontrivial, is a spherical distribution. 
Choosing $Z\in\Gamma(\Delta_0)$ and $T\in\Gamma(\Delta_2)$, then \eqref{derR} gives
$$
(I-R_f)\nabla_ZT=0.
$$
This implies that $\Delta_0\oplus\Delta_2$ is a totally geodesic distribution.
The same argument shows that $\Delta_0\oplus\Delta_1$ is also totally geodesic.
Summarizing, we have an orthogonal decomposition
$$
TM=\Delta_0\oplus\Delta_1\oplus\Delta_2,
$$
with $\Delta_0\oplus\Delta_i$ totally geodesic and $\Delta_i$ spherical, $i=1,2$.
It follows from Hiepko's theorem \cite{hiepko} (see also Theorem 10.4 in \cite{DT}) that (locally) the Riemannian manifold $M^n$ is (isometric to) a warped product
$$
M^n=M_0^3\times_{\rho_1}M^{\ell_1}_1\times_{\rho_2}M^{\ell_2}_2,
$$
where $\ell_1+\ell_2=n-3$, $\ell_i\geq 0$ and $\rho_i\in C^\infty(M_0)$, $i=1,2$.

Observe that $\Delta_1\subset\ker(R_f)\cap\ker(R_g)$ and $\Delta_2\subset \ker(I-R_f)\cap\ker(I-R_g)$. Thus, we can apply Proposition~\ref{propmulti} and conclude that $f$ and $g$ are  rotational or doubly rotational hypersurfaces. In the latter case,
if we fix $\bar{x}_0\in M_0$, the immersions $f(\bar{x}_0,\cdot,\cdot)\colon M_1\times M_2\to\Hy^k\times\Sf^{n-k+1}$ and $g(\bar{x}_0,\cdot,\cdot)\colon M_1\times M_2\to\Hy^k\times\Sf^{n-k+1}$ are isometric extrinsic products of umbilical immersions, say $f(\bar{x}_0,\cdot,\cdot)=h_1^f\times h_2^f$ and $g(\bar{x}_0,\cdot,\cdot)=h_1^g\times h_2^g$. 
Moreover $h_i^f$ is congruent to $h_i^g$, $i=1,2$. 
Hence $f(\bar{x}_0,\cdot,\cdot)$ and $g(\bar{x}_0,\cdot,\cdot)$ are congruent in $\Hy^k\times\Sf^{n-k+1}$. Therefore, we can assume that $f$ and $g$ are doubly rotational hypersurfaces based on the same extrinsic product of umbilical immersions. Similarly if $f$ and $g$ are both rotational hypersurfaces.
 We have seen that $R_f|_{\Delta^\perp}$ and $R_g|_{\Delta^\perp}$ have $0$ and $1$ as eigenvalues. Hence, the profiles of $f$ and $g$ are necessarily isometric hypersurfaces $f_0, g_0\colon M_0^3\to \Hy^2\times\Sf^2$. 
\qed
\vspace{1ex}

In the sequel, we assume that $S_f\neq S_g$ everywhere, and we study separately the cases in which $A_f-A_g$ has rank one or two.

\begin{proposition}\label{confrul}
Let $f\colon M^n\to \Hy^k\times \Sf^{n-k+1}$, $n\geq 6$, $2\leq k\leq n-1$, be a truly deformable hypersurface that does not split. Assume that $f$ admits a true deformation $g\colon M^n\to \Hy^k\times \Sf^{n-k+1}$ such that the common eigenbundle $\Delta$ of $f$ and $g$ and the common principal curvature $\lambda$ of $A_f|_{\Delta}$ and $A_g|_{\Delta}$, given by Proposition \ref{commoneig2},  satisfy  $\lambda\neq 0$ and $\dim \Delta=n-2$. If, in addition, $S_f\neq S_g$ and $A_f-A_g$ has rank one everywhere, then $f$ and $g$ are conformally ruled with the same rulings.  
\end{proposition}
\proof
Following the previous notation, without loss of generality, we may assume that $r_2=0=\lambda_2$. 
Then, equations  \eqref{TXY} and \eqref{deltaY} imply that
\be\label{rul1}
\<\nabla_TX,Y\>=0
\ee
and 
$$
(S_f-S_g)(Y)=0,
$$
respectively.
Taking the difference of \eqref{derR} for $f$ and $g$ evaluated in $Y$, we obtain
$$
-(R_f-R_g)\nabla_YY=\<A_fY,Y\>(\xi_f-\xi_g),
$$
which implies that
\be\label{rul2}
-r_1\<\nabla_YY,X\>=\<A_fY,Y\>(S_f-S_g)(X).
\ee
The difference of \eqref{codazzi} for $f$ and $g$, evaluated in $X$ and $Y$, gives
$$
-(A_f-A_g)\nabla_XY-Y(\lambda_1)X-\lambda_1\nabla_YX=(S_f-S_g)(X)Y.
$$
Taking the $Y$-component of both sides of the above equation yields
\be\label{YYX}
-\lambda_1\<\nabla_YX,Y\>=(S_f-S_g)(X),
\ee
which, together with \ref{deltaX}, implies that
$$
\<\nabla_YX,Y\>=\<\nabla_TX,T\>.
$$
It follows from \eqref{YTX}, \eqref{rul1} and the previous equation that $\spa\{Y\}\oplus\Delta$ is an umbilical distribution.
Finally, by \ref{deltaX} and \eqref{rul2} we have
$$
\lambda=\<A_fY,Y\>,
$$
hence the hypersurfaces $f$ and $g$ are conformally ruled with the leaves of $L=\spa\{Y\}\oplus\Delta$ as common rulings.
\qed

\vspace{1ex}

\begin{proposition}\label{multi2}
Let $f\colon M^n\to \Hy^k\times \Sf^{n-k+1}$, $n\geq 6$, $2\leq k\leq n-1$, be a truly deformable hypersurface that does not split in any open subset. Assume that $f$ admits a true deformation $g\colon M^n\to \Hy^k\times \Sf^{n-k+1}$ such that $S_f\neq S_g$ everywhere, and such that the common principal curvature $\lambda$ of $f$ and $g$ given by Proposition \ref{commoneig2}, together with $\Delta=\ker(A_f-\lambda I)\cap \ker(A_g-\lambda I)$,  satisfy  $\lambda\neq 0$, $\dim \Delta=n-2$ and $\rank(A_f-A_g)=2$.
Then the following holds:
\begin{itemize}
\item [(i)]  on each connected component $\mathcal{V}$ of the interior of the subset where $\Delta\subset\ker S_f$, 
either $k=2$ or  $k=n-1$, and $f|_\mathcal{V}$ and $g|_\mathcal{V}$ are locally rotational hypersurfaces  of types $(I)$ or $(II)$, respectively,  with isometric surfaces  $f_0,g_0\colon M_0^2\to \mathbb{H}^\ell\times \mathbb{S}^{3-\ell}$ as profiles, respectively, where $\ell=2$ or $\ell=1$;
 \item [(ii)]  on each connected component $\mathcal{V}$ of the open subset where $\Delta\not \subset\ker S$, either $f|_\mathcal{V}$ and $g|_\mathcal{V}$ are locally doubly rotational hypersurfaces, with isometric three-dimensional hypersurfaces $f_0,g_0\colon M_0^3\to \mathbb{H}^\ell\times \mathbb{S}^{4-\ell}$, $\ell\in \{1,2,3\}$, as profiles, 
 or $2\leq k\leq 3$ (respectively, $n-2\leq k\leq n-1$) and $f, g$ are rotational hypersurfaces  of type $(I)$ (respectively, type $(II)$) with three-dimensional hypersurfaces $f_0, g_0\colon M_0^3\to \mathbb{H}^\ell\times \mathbb{S}^{4-\ell}$, $\ell\in \{2,3\}$ (respectively, $\ell\in \{1,2\}$), as profiles.
\end{itemize}
\end{proposition}

\proof 
With the notations introduced before Proposition \ref{sameS}, the assumption that $\rank(A_f-A_g)=2$ is equivalent to $\lambda_1\lambda_2\neq 0$.

First, assume that $\Delta\subset\ker S_f$. In this case,
Eq. \eqref{derS} evaluated in $T\in\Gamma(\Delta)$ gives
$$
-S_f(\nabla_TT)=\lambda(t_f-\<R_fT,T\>).
$$
Since $\Delta$ is an umbilical distribution,  the left-hand side of the previous equation does not depend on the choice $T$, hence the same holds for the right-hand side. Observe that, by the assumption $n\geq 6$, the common eigenbundle $\Delta$ of $f$ and $g$ intersects  $\ker(I-R_f)$ or $\ker R_f$. This and the previous equation show that $\Delta$ must be contained in one of these kernels, because, otherwise,  $\<R_fT, T\>$ would depend on $T$. Thus, from our assumptions, the possible values of $k$ lie in $\{2,n-1\}$.
Then,  from \eqref{XXT}, \eqref{XTY}, \eqref{YYT} and \eqref{YTX} it follows that $\Delta^\perp$ is a totally geodesic distribution.

Let $\delta$ denote the mean curvature vector field of $\Delta$, that is, $(\nabla_TT)_{\Delta^\perp}=\delta$ for any $T\in\Gamma(\Delta)$ of unit length.
The difference of \eqref{derS} for $f$ and $g$, evaluated in $T$ and $X$, gives
\be\label{TSX}
T((S_f-S_g)(X))-(S_f-S_g)(\nabla_TX)=0.
\ee
If $r_1\neq r_2$ (or equivalently, $\lambda_1\neq\lambda_2$), then
\eqref{TXY} implies that $\<\nabla_TX,Y\>=0$. 
If $r_1=r_2$, we take $X$ and $Y$ such that $Y\in\Gamma(\ker (S_f-S_g)\cap\Delta^\perp)$. Then, Eq. \eqref{derS} for $f$ and $g$, evaluated in $T$ and $Y$, gives
$$
(S_f-S_g)(\nabla_TY)=0,
$$
which implies that $\<\nabla_TY,X\>=0$ also in this case.
Hence $\nabla_TX\in\Gamma(\Delta)$, and therefore \eqref{TSX} becomes
$$
T((S_f-S_g)(X))=0.
$$
Taking the derivative of \eqref{deltaX} with respect to $T$, keeping in mind \eqref{Tr1} and the equation above, we obtain
$$
\<\nabla_T\delta,X\>=0.
$$
In a similar way we see that $\<\nabla_T\delta,Y\>=0$, hence $(\nabla_T\delta)_{\Delta^\perp}=0$.
Therefore, $\Delta$ is a spherical distribution, and $\Delta^\perp$ is totally geodesic.

It follows from Hiepko's theorem  that $M^n$ is locally a warped product $M^n=M_0^2\times _\rho M_1^{n-2}$, whose product net is $(\Delta, \Delta^\perp)$. 
Since $\Delta$ is an eigenspace of $A_f$ and $A_g$ and $\Delta$ is contained in $\ker(R_f)\cap\ker(R_g)$ or $\ker(I-R_f)\cap\ker(I-R_g)$, we are in conditions for applying Proposition \ref{propmulti}. Then $f$ and $g$ are locally rotational hypersurfaces of types $(I)$ or $(II)$, with isometric surfaces  $f_0,g_0\colon M_0^2\to \mathbb{H}^\ell\times \mathbb{S}^{3-\ell}$ as profiles, where $\ell=2$ or $\ell=1$ respectively.
\vspace{2ex}

Let us now assume that $\Delta\not\subset\ker S_f$,  and denote $\bar{\Delta}=\ker S_f\cap\Delta=\ker S_g\cap\Delta$.
If $\lambda_1-\lambda_2\neq0$, then \eqref{XTY}, \eqref{TXY} and \eqref{YTX} give 
\be\label{nabXTY}
\<\nabla_XT,Y\>=\<\nabla_YT,X\>=\<\nabla_TX,Y\>=0
\ee
for any $T\in\Gamma(\Delta)$.
If $\lambda_1=\lambda_2$, we take $X$ and $Y$ such that $Y\in\Gamma(\ker(S_f-S_g)\cap\Delta^\perp)$. Then, equation  \eqref{derS} for $f$ and $g$, evaluated in $T$ and $Y$, gives
$$
(S_f-S_g)(\nabla_TY)=0,
$$
which implies that $\<\nabla_TY,X\>=0$ for any $T\in\Gamma(\Delta)$. Hence, equation  \eqref{nabXTY} still holds in this case.

Since $r_i=-\lambda\lambda_i$, $i=1,2$, then \eqref{XXT} and \eqref{YYT} give 
\begin{equation}\label{XXYY}
    (\nabla_XX)_{\Delta}=(\nabla_YY)_{\Delta},
\end{equation}
and this vector field is orthogonal to $\bar{\Delta}$. Let $Z\in\Gamma(\Delta\cap\bar{\Delta}^\perp)$ have unit length, and pick $T\in\Gamma(\bar{\Delta})$ also with unit length. Taking the difference of \eqref{derS} for $f$ and $g$, evaluated in $X$ and $T$, we have
$$
-(S_f-S_g)(\nabla_XT)=-\lambda_1\<X,R_fT\>.
$$
Since $T\in\Gamma(\bar{\Delta})$, from \eqref{nabXTY} it follows that $\nabla_XT\in\Gamma(\Delta)$, and hence $(S_f-S_g)(\nabla_XT)=0$. Thus, the previous equation gives
$$
\<X,R_fT\>=0,
$$
and, similarly,
$$
\<Y,R_fT\>=0,
$$
which implies that $R_f(\bar{\Delta})\subset\Delta$. Recall from \eqref{kerS} that $\ker S_f$ is invariant by $R_f$, hence $\bar{\Delta}$ is invariant by $R_f$ (and therefore by $R_g$). Denote $\Delta_0=\bar{\Delta}^\perp$, $\Delta_1=\ker(R_f)\cap\Delta$, and $\Delta_2=\ker(I-R_f)\cap\Delta$. Then we have an orthogonal decomposition
$$
TM=\Delta_0\oplus\bar{\Delta}=\Delta_0\oplus\Delta_1\oplus\Delta_2.
$$
From \eqref{derR} we see that
$$
\nabla_XR_fZ-R_f\nabla_XZ=S_f(Z)A_fX,
$$
whose $T$-component is
$$
\<R_fZ,\nabla_XT\>+\<\nabla_XZ,R_fT\>=0.
$$
Since $\nabla_XT$ has no component in $\Delta^\perp$, then
$$
\<R_fZ,Z\>\<\nabla_XT,Z\>+\<\nabla_XZ,R_fT\>=0.
$$
If $T=T_2\in\Delta_2$, then 
$$ 
\<\nabla_XT_2,Z\>(\<R_fZ,Z\>-1)=0,
$$
while, if $T=T_1\in\Delta_1$, then
$
\<\nabla_XT_1,Z\>\<R_fZ,Z\>=0.
$
Since $0<\<R_fZ,Z\><1$, for $S_f(Z)\neq 0$, then
\be\label{XZ}
\nabla_XZ\in\Gamma(\Delta_0),
\ee
and the same arguments give 
\be\label{YZ}
\nabla_YZ\in\Gamma(\Delta_0).
\ee
From \eqref{derS}, evaluated in $Z$ and $T$, we obtain
$
S_f(\nabla_ZT)=\lambda\<Z,R_fT\>=0,
$
which, together with the umbilicity of $\Delta$, implies that
\be\label{ZZ}
\<\nabla_ZT,Z\>=0.
\ee
Summarizing, Eqs. \eqref{nabXTY}, \eqref{XXYY}, \eqref{XZ}, \eqref{YZ}, \eqref{ZZ}, the definition of $Z$ and the umbilicity of $\Delta$ imply that $\Delta_0$ is a totally geodesic distribution. 

Let $T\in\Gamma(\bar{\Delta})$ and $T_i\in\Gamma(\Delta_i)$, $i=1,2$, have unit length. Then, Eq. \eqref{derR} gives
\be\label{nabTT1}
(I-R_f)\nabla_TT_2=\lambda\<T,T_2\>\xi_f
\ee
and 
\be\label{nabTT2}
R_f\nabla_TT_1=-\lambda\<T,T_1\>\xi_f.
\ee
In particular,
\be\label{TTi}
\mbox{if}\;\; \<T,T_i\>=0 \;\; \mbox{then}\;\; \nabla_TT_i\in\Delta_i.
\ee
If $\delta$ is the mean curvature vector field of $\Delta$, then \eqref{nabTT1} and \eqref{TTi} imply that
$$
(I-R_f)\nabla_{T_2}T_2=(I-R_f)\delta+\<\nabla_{T_2}T_2,Z\>(I-R_f)Z=\lambda\xi_f.
$$
This expression shows that $\<\nabla_{T_2}T_2, Z\>$ does not depend on the choice of $T_2$. 
Hence, $\Delta_2$ is an umbilical distribution.
A similar argument, using \eqref{nabTT2}, shows that $\Delta_1$ is also an umbilical distribution. Let $\delta+\delta_i$ denote the mean curvature vector field of $\Delta_i$, $i=1,2$.
It follows from \eqref{derR} that $(I-R_f)\nabla_XT_2=0=(I-R_f)\nabla_YT_2$, and that $R_f\nabla_XT_1=0=R_f\nabla_YT_1$. This, together with \eqref{TTi}, the umbilicity of $\Delta$ and the fact that $\Delta_0$ is totally geodesic imply that $\Delta_0\oplus\Delta_i$ is a totally geodesic distribution for $i=1,2$.

We claim that $\Delta_i$ is a spherical distribution for each $i=1,2$.
The difference of \eqref{derS} for $f$ and $g$, evaluated in $T\in\Gamma(\bar{\Delta})$ and $X$, gives
$$
T((S_f-S_g)(X))-(S_f-S_g)(\nabla_TX)=0.
$$
Using \eqref{nabXTY} and the umbilicity of $\Delta$, we have 
$
T((S_f-S_g)(X))=0.
$
Then, taking the derivative of \eqref{deltaX} in the direction of $T$, we obtain
$
\<\nabla_T\delta,X\>=0.
$
In a similar way we obtain $T((S_f-S_g)(Y))=0$. Taking the derivative of \eqref{deltaY} in the direction of $T$, we see that
$
\<\nabla_T\delta,Y\>=0,
$
which, together with the fact that $\Delta$ is umbilical, implies that 
$$
(\nabla_T\delta)_{\Delta_0}=0
$$
for any $T\in\bar{\Delta}$.
To prove the claim, it suffices to show that $\<\nabla_{T_i}\delta_i,Z\>=0$, $i=1,2$. 
By  \eqref{derS}, evaluated in $T_2$, we have  
$$
S_f(\delta+\delta_2)=\lambda(1-t_f),
$$
and its derivative in the direction of $T_2$ gives
$$
T_2(S_f(\delta+\delta_2))=-\lambda T_2(t_f),
$$
where we have used that $T_2(\lambda)=S_f(T_2)=0$.
On the other hand, by \eqref{derS},
$$
T_2(S_f(\delta+\delta_2))=S_f(\nabla_{T_2}(\delta+\delta_2)).
$$
Since  $(\nabla_{T_2}\delta)_{\Delta_0}=0$, then $S_f(\nabla_{T_2}(\delta+\delta_2))=S_f(\nabla_{T_2}\delta_2)$. It follows from the previous expressions and the fact that $\Delta$ is umbilical that
$$
\<\nabla_{T_2}\delta_2,Z\>S_f(Z)=-\lambda T_2(t_f)=0,
$$
where the last equality follows from \eqref{dert}. This proves that $(\nabla_{T_2}(\delta+\delta_2))_{\Delta_0}=0$, and hence that $\Delta_2$ is a spherical distribution.
Following the same arguments with $T_1\in\Gamma(\Delta_1)$, we see that $\Delta_1$ is also a spherical distribution, which proves the claim.

Summarizing, in any case we have an orthogonal decomposition of $TM$ as  $TM=\Delta_0\oplus\Delta_1\oplus\Delta_2$, where $\Delta_0$ and $\Delta_0\oplus\Delta_i$ are totally geodesic, and $\Delta_i$ is a spherical distribution, $i=1,2$. Here  $\Delta_i$ could be trivial for some $i$, depending on the value of $k$.
It follows from Hiepko's theorem  that (locally) the Riemannian manifold $M^n$ is (isometric to) a warped product
$$
M^n=M_0^3\times_{\rho_1}M^{\ell_1}_1\times_{\rho_2}M^{\ell_2}_2,
$$
where $\ell_1+\ell_2=n-3$, $\ell_i\geq 0$ and $\rho_i\in C^\infty(M_0)$, $i=1,2$.

Observe that the subbundles $\Delta_1\subset \ker(R_f)\cap\ker(R_g)$ and $\Delta_2\subset\ker(I-R_f)\cap\ker(I-R_g)$ and that $\Delta$ is a common eigenspace of $A_f$ and $A_g$. Thus, from Proposition \ref{propmulti} it follows that $f$ and $g$ are doubly rotational hypersurfaces. As before, we can assume that $f$ and $g$ are based on the same extrinsic product of umbilical immersions. Therefore, they are given in terms of isometric profiles. If some $\Delta_i$ is trivial, then $f$ and $g$ are rotational hypersurfaces of spherical or hyperbolic type given in terms of isometric profiles.
\qed

\begin{remark}\emph{We point out that the hypersurfaces in the previous proposition, when seen as hypersurfaces of $\R^{n+1}\setminus\R^{k-1}$ endowed with the metric induced by the conformal diffeomorphism between $\R^{n+1}\setminus\R^{k-1}$ and $\mathbb{H}^k\times \mathbb{S}^{n-k+1}$, are conformally surface-like hypersurfaces. This follows from the facts that the distribution $\Delta^\perp$ is umbilical, and that the umbilicity of $\Delta^\perp$ was shown in \cite{DFT} to characterize conformally surface-like hypersurfaces.}
\end{remark}

\subsubsection[Case lambda not 0 and dimDelta=n-1]{Case $\lambda\neq 0$ and $\dim\Delta=n-1$. }

When $\dim \Delta=n-1$ and $\lambda\neq 0$ for the eigenbundle $\Delta$ and the principal curvature $\lambda$ associated with a truly deformable hypersurface and a true deformation of it, we have the following result.

\begin{proposition}\label{multi1}
Let $f\colon M^n\to \Hy^k\times \Sf^{n-k+1}$, $n\geq 6$, $2\leq k\leq n-1$, be a truly deformable hypersurface that does not split in any open subset. Assume that $f$ admits a true deformation $g\colon M^n\to \Hy^k\times \Sf^{n-k+1}$ such that the common eigenbundle $\Delta$ of $f$ and $g$ and the common principal curvature $\lambda$ of $A_f|_{\Delta}$ and $A_g|_{\Delta}$, given by Proposition \ref{commoneig2}, satisfy  $\lambda\neq 0$ and $\dim \Delta=n-1$ at every point. 
Then the subset where $S_f=S_g$ or $\Delta=\ker S_f$ has empty interior and, on each connected component $\mathcal{V}$ of the open and dense subset where $S_f\neq S_g$ and $\Delta\not \subset\ker S$, one of the following holds:
\begin{itemize}
    \item [(i)]  $k=2$  (respectively, $k=n-1$), and $f$, $g$ are locally rotational hypersurfaces of type $(I)$ (respectivey, type $(II)$, whose profiles are isometric surfaces $f_0,g_0\colon M_0^2\to \mathbb{H}^\ell\times \mathbb{S}^{3-\ell}$, with $\ell=2$ (respectivey, $\ell=1$);

    \item [(ii)] $f$ and $g$ are locally doubly rotational hypersurfaces whose profiles are isometric surfaces $f_0,g_0\colon M_0^2\to \mathbb{H}^\ell\times \mathbb{S}^{3-\ell}$, $\ell\in \{1,2\}$.
\end{itemize}
 
\end{proposition}
\proof 
First, observe that $R_f\neq R_g$. Otherwise, from \eqref{reltens3} we would have $A_f=A_g$, hence $S_f=S_g$ by \eqref{derR}. This would imply that $f$ and $g$ are congruent by Proposition \ref{cong}, a contradiction. 

We claim that $S_f\neq S_g$ (or equivalently, $\xi_f\neq \xi_g$). Assume otherwise. Recall that \eqref{codazzi} implies that $S_f(T)=T(\lambda)=S_g(T)$ for any $T\in\Gamma(\Delta)$.
Let $X\in\Gamma(\Delta^\perp)$ and $T\in\Gamma(\Delta)$ have unit length.
The difference of \eqref{derS} for $f$ and $g$, evaluated in $X$ and $T$, gives
$$
(S_f-S_g)(\nabla_XT)=\lambda_1\<X,R_fT\>,
$$
where $(A_f-A_g)X=\lambda_1X$. Then $\Delta$ is invariant by $R_f$ and $R_g$, or equivalently, $X$ is a common eigenvector of $R_f$ and $R_g$.
From the difference of \eqref{derS} for $f$ and $g$, evaluated in $T$, we obtain
$
\lambda(t_f-t_g)=0.
$
Hence $t_f=t_g$, and therefore $r_f=r_g$. This implies that, at any point $x$, $\xi_f(x)=\xi_g(x)\in\ker(R_f-R_g)(x)=\Delta(x)$.
Then $X$ is an eigenvector corresponding to the eigenvalues $0$ or $1$ of $R_f$ and $R_g$. Since $R_f|_\Delta=R_g|_\Delta$,  from \eqref{dimkerR} we see that $R_fX=R_gX$, which is a contradiction that proves our claim.

 Now we show that $\ker S_f \neq \Delta$. Otherwise, also 
 $\ker S_g=\Delta$. Let $\delta$ be the mean curvature vector field of $\Delta$. From \eqref{derS}, evaluated in $T$, we obtain
$$
S_f(\delta)=\lambda(\<T,R_fT\>-t_f).
$$
Since the left-hand side of the previous equation does not depend on the choice of $T$, then $\Delta$ must be contained in $\ker(I-R_f)$ or in $\ker R_f$. This already implies $k\in\{1,n\}$, which is ruled out by our assumptions.

Thus $\Delta\not\subset\ker S_f$. Denote $\bar{\Delta}=\Delta\cap\ker S_f=\Delta\cap\ker S_g$, and let $Y\in\Gamma(\Delta\cap\bar{\Delta}^\perp)$ have unit length. 
We follow the same arguments as in the proof of Proposition \ref{multi2}. 
For instance, it follows from \eqref{derR} that
$
\<\nabla_XX,T\>=0
$
for any $T\in\Gamma(\bar{\Delta})$. Then, the difference of \eqref{derS} for $f$ and $g$, evaluated in $X$ and $T$, gives $R_f(\bar{\Delta})=R_g(\bar{\Delta})\subset\Delta$. This, together with the fact that $\ker S_f$ and $\ker S_g$ are invariant by $R_f$ and $R_g$, respectively, implies that $\bar{\Delta}$ is invariant by $R_f$ and $R_g$. 
Using this fact, we see from \eqref{derS}, evaluated in $Y$ and $T$, that $\<\nabla_YY,T\>=0$. Define the distribution $\Delta_0=\bar{\Delta}^\perp=\spa\{X,Y\}$. Then, we obtain an orthogonal decomposition
$$
TM=\Delta_0\oplus\bar{\Delta}=\Delta_0\oplus\Delta_1\oplus\Delta_2,
$$
where $\Delta_1=\Delta\cap\ker(R_f)$ and $\Delta_2=\Delta\cap\ker(I-R_f)$. Note that $\dim \Delta_1\geq k-2$ and $\dim \Delta_2\geq n-k-1$. Thus $\Delta_1$ (respectively, $\Delta_2) $ may be trivial if (and only if) $k=2$ (respectively, $k=n-1$). 
Taking the $T_i$-component of $\eqref{derR}$ evaluated in $X$ and $Y$, where $T_i\in\Gamma(\Delta_i)$, $i=1,2$, we see that 
$$
\<R_fY,\nabla_XT_i\>+\<\nabla_XY,R_fT_i\>=0.
$$
Since $\nabla_XT_i$ has no component in the direction of $X$, we have
$$
\<R_fY,Y\>\<\nabla_XT_i,Y\>+\<\nabla_XY,R_fT_i\>=0.
$$
Hence 
$$
\<\nabla_XT_1,Y\>\<R_fY,Y\>=0=\<\nabla_XT_2,Y\>(\<R_fY,Y\>-1).
$$
Taking into account that $0<\<R_fY,Y\><1$, we obtain that $\<\nabla_XT_i,Y\>=0$, $i=1,2$, or equivalently, that $\nabla_XY\in\Gamma(\Delta_0)$.
It follows from \eqref{derS}, evaluated in $Y$ and $T$, that $\nabla_YY\in\Gamma(\Delta_0)$, hence $\Delta_0$ is totally geodesic.

The same arguments used in the proof of Proposition \ref{multi2} show that $\Delta_i$ is a spherical distribution and that $\Delta_0\oplus\Delta_i$ is totally geodesic, $i=1,2$. 
For simplicity, we omit repeating the details.
Therefore, in any case, $M^n$ is (locally) a warped product
$
M^n=M_0^2\times M_1^{\ell_1}\times M_2^{\ell_2}
$,
with $\ell_1+\ell_2=n-2$ and $\ell_i\geq 0$, $i=1,2$.
As before, we have $\Delta_1\subset\ker(R_f)\cap\ker(R_g)$ and $\Delta_2\subset\ker(I-R_f)\cap\ker(I-R_g)$, and they are contained in eigenspaces of $A_f$ and $A_g$. Thus, arguing as before, by Proposition \ref{propmulti}, they are locally doubly rotational hypersurfaces whose profiles are isometric surfaces in $\mathbb{H}^\ell\times \mathbb{S}^{3-\ell}$, $\ell\in \{1,2\}$. 
If $\Delta_1$ (respectively, $\Delta_2$) is trivial, then $k=2$ (respectively, $k=n-1$), and $f$, $g$ are rotational hypersurfaces of type $(I)$ (respectively, type $(II)$)  whose profiles are isometric surfaces in $\mathbb{H}^{2}\times \mathbb{S}^{1}$ (respectively, $\mathbb{H}^{1}\times \mathbb{S}^{2}$.
\qed

\subsection{Proof of Theorem \ref{main}}

The statement of Theorem \ref{main} will follow by putting together the several propositions in the preceding subsection. 

    Let $U_0$ be the interior of the subset of $M^n$ where the one-form $S_f$ vanishes. Then $f$ is locally as in $(i)$ or $(ii)$ in $U_0$ by Proposition \ref{splits}.
    
    Let $W_0$ be the interior of the subset where $\nu_f\leq n-2$, and let $V_0=W_0\cap(M^n\setminus \bar{U}_0)$. 
    By Proposition \ref{commoneig2}, if $g\colon M^n\to \Hy^{k}\times\Sf^{n-k+1}$ is a true deformation of $f$, then $A_f$ and $A_g$  share a common eigenbundle $\Delta$ of rank $n-1$ or $n-2$ on each connected component of an open and dense subset of $V_0$, the shape operators  $A_f$ and $A_g$ share the same (up to sign) eigenvalue $\lambda$ on $\Delta$ and $R_f|_{\Delta}=R_g|_{\Delta}$.
    
    Let $V_1$ and $V_2$ be the open subsets of $V_0$ where $\lambda\neq 0$ and $\dim \Delta=n-2$, respectively.
    By Proposition \ref{ruled}, if $f$ is neither as in $(i)$ nor as in $(ii)$, then it is as in $(iii)$ on $U_1=V_2\cap (V_0\setminus \bar{V}_1)$. 
    On the other hand, it follows from Proposition \ref{multi1} that $f$ is locally as in $(v)$ or $(vii)$ on $U_2=V_1\cap (V_0\setminus \bar{V}_2)$.
    
    Now set $V_3=V_1\cap V_2$, and let
    $W_i$, $1\leq i\leq 3$, be the open subsets of $V_3$ where $S_f\neq S_g$, $\rank (A_f-A_g)=2$ and $\Delta\not \subset \ker S_f$, respectively.  By Proposition \ref{confrul}, 
    $f$ is as in $(iv)$ on $U_3=W_1\cap( V_3\setminus \bar{W}_2)$. Proposition \ref{sameS} shows that  $U_4=(V_3\setminus \bar{W}_1)\cap {W}_3$ is dense on $V_3\setminus \bar{W}_1$, and that $f$ is locally as in $(vi)$ or $(viii)$ with $\ell=2$  on $U_4$.
   Finally, it follows from item $(i)$ of Proposition \ref{multi2} that $f$ is locally as in $(v)$ on $U_5=W_1\cap W_2 \cap (V_3\setminus \bar{W}_3)$, whereas we have from item $(ii)$ in that proposition that $f$ is locally as in $(vi)$ or $(viii)$ on $U_6=W_1\cap W_2 \cap W_3$. 
    
    
      The statement then holds on the open dense subset $\mathcal{U}=\cup_{i=1}^6 U_i$.\qed
\section{Towards the converse of Theorem \ref{main}}\label{Sec:final}
In this final section, we prove some partial results and pose some problems regarding the converse of Theorem \ref{main}. 

   First, it is clear that any hypersurface as in items $(i)$ and $(ii)$ is truly deformable; its true deformations are given by isometric deformations of $f_1\colon N^{k-1}\to \mathbb{H}^k$ (in item $(i)$) or $f_2\colon N^{n-k}\to \mathbb{S}^{n-k+1}$ (in item $(ii)$).  
   
     Proposition \ref{ruledconv} below shows that any simply connected ruled hypersurface is truly deformable, with a large set of true isometric deformations.
     
     \begin{proposition}\label{ruledconv}
Any simply connected ruled hypersurface $f\colon M^n\to \Hy^k\times \Sf^{n-k+1}$, $n\geq 6$, $2\leq k\leq n-1$, with index of relative nullity $n-2$ at any point is truly deformable. Moreover, if it does not split on any open subset, then its true deformations are in one-to-one correspondence with the set of smooth functions over an interval, and all of them are also ruled with the same rulings.
\end{proposition}
\proof Let $L$ be the totally geodesic distribution of rank $(n-1)$ of $M^n$ whose leaves are mapped by $f$ into totally geodesic submanifolds of $\Hy^k\times \Sf^{n-k+1}$. We may assume that $f$ does not split on any open subset.
By Proposition \ref{umbeig}, the distribution $\Delta=\ker A_f$ is umbilical. Let $\{X,Y\}$ be an orthonormal frame of $\Delta^\perp$ such that $X\in\Gamma(L^\perp)$ and $Y\in\Gamma(L\cap\Delta^\perp)$.
The matrix of $A_f$, restricted to $\Delta^\perp$, with respect to $\{X, Y\}$, is
$$
A_f=\begin{bmatrix}
\lambda & \mu\\
\mu & 0
\end{bmatrix}.
$$
The fact that $L$ is totally geodesic, together with \eqref{codazzi} evaluated in $Y$ and $T\in\Gamma(\Delta)$, gives
$L=\ker S_f$. 
Our next goal is to build a symmetric tensor $A'\in\Gamma(\End(TM))$ distinct from $\pm A_f$  that, together with $R_f$ and $S_f$, satisfies the Gauss and Codazzi equations \eqref{Gauss} and \eqref{codazzi}. Define $J\in\Gamma(\End(TM))$ by $JX=Y$ and $J|_{\ker S_f}=0$.
We look for a tensor of the form
$$
A'=A_f+\theta A_fJ,
$$
where $\theta$, constant along the leaves of $\Delta$, is a function to be determined.
The matrix of $\theta A_fJ$, restricted to $\Delta^\perp$, in terms of $\{X, Y\}$, is
$$
\theta A_fJ=\begin{bmatrix}
\theta\mu & 0\\
0 & 0
\end{bmatrix}.
$$
 Since we are asking $\theta$ to be constant along the leaves of $\Delta$, it follows from \eqref{codazzi} for $A_f$ that \eqref{codazzi} also holds for $A'$ when one of the vector fields is $T\in\Gamma(\Delta)$.
Thus \eqref{codazzi}, for $A'$ and $S_f$, is equivalent to the equation
\be\label{codrul}
Y(\log \theta\mu)=\<\nabla_XX,Y\>.
\ee
If we take an arbitrary function as initial condition along one maximal integral curve of $X$, then there exists a unique function $\theta$ such that $\theta\mu$ is a solution of \eqref{codrul}  and $T(\theta)=0$ for all $T\in\Gamma(\Delta)$.
Hence, for any such initial condition, we obtain a tensor $A'$ that, together with $S_f$, satisfies \eqref{codazzi}.

Now notice that $\theta A_fJ=A'-A_f$ has rank one and vanishes on $\ker S_f$. This, together with $A_fY=\mu X$, implies that \eqref{Gauss} holds for $A'$ and $R_f$.
Therefore, the tensors $A'$, $R_f$ and $S_f$ satisfy all the compatibility conditions \eqref{derR}, \eqref{derS2}, \eqref{dert}, \eqref{Gauss} and \eqref{codazzi}. It follows from Theorem \ref{exist} that there  exists an isometric immersion $g\colon M^n\to  \Hy^k\times\Sf^{n-k+1}$ that is a true deformation of $f$.\vspace{1ex}
\qed

Determining which conformally ruled hypersurfaces are truly deformable proved to be a challenging task. 
Let $f\colon M^n\to\Hy^k\times\Sf^{n-k+1}$ be a conformally ruled hypersurface and let $X\in\mathfrak{X}(M)$ have unit length and be orthogonal to the rulings $L$.
If $f$ has a true deformation $g$, we saw in the proof of Proposition \ref{confrul} that $R_f|_L=R_g|_L$. Hence $R_g$ is completely determined by a function $r_1\in C^\infty(M)$ from the expression $(R_f-r_1I)X=R_gX$.
This function would also determine the tensors $A_g$ and $S_g$ by \eqref{reltens3} and \eqref{YYX}. In the case of Proposition \ref{ruledconv}, the only restriction was equation \ref{codazzi}. However, in this case we have the additional restrictions given by \eqref{derR} and \eqref{derS2}. For instance, following the notation in the proof of Proposition \ref{confrul}, if we evaluate \eqref{derR} in $X$ (in both slots) and $Y$ (in both slots), we obtain the following necessary condition for the existence of a true deformation:
$$
S_f(Y)+\lambda\<\nabla_XX,Y\>-\mu\<\nabla_YX,Y\>=0.
$$
This already points in the direction that not every conformally ruled hypersurface admits a true deformation. 

Let us now focus on the problem of deciding which of the rotational and doubly rotational hypersurfaces of $\Hy^k\times\Sf^{n-k+1}$in Theorem \ref{main} are truly deformable. Let $f\colon M^n\to\Hy^k\times\Sf^{n-k+1}$ be a truly deformable doubly rotational hypersurface with true deformation $g$. We have seen that $g$ is also doubly rotational, and we can assume that it is based on the same extrinsic product of umbilical immersions. We know that their profiles are isometric, but it is not any pair of isometric profiles that produces true deformations. 

Let $f_0, g_0\colon M_0\to \Hy^{\ell_1}\times\Sf^{\ell_2}$, $\ell_1+\ell_2\leq 4$, be the profiles of $f$ and $g$ respectively. 
Fix bases $\{e_i\}_{0\leq i\leq\ell_1}$ and $\{d_j\}_{0\leq j\leq \ell_2}$ of $\Les^{\ell_1+1}$ and $\R^{\ell_2+1}$ such that
$\{e_i\}_{0\leq i\leq\ell_1}$ is orthonormal or pseudo-orthonormal depending on whether we are in the elliptic, hyperbolic or parabolic case and $\{d_j\}_{0\leq j\leq \ell_2}$ is orthonormal.
Then the warping functions of $M=M_0\times_{\rho_1}M_1\times_{\rho_2}M_2$ are given by \eqref{rho1} and \eqref{rho2}. 
This provides the first necessary condition on the profiles, say,
\be\label{isom11}  
\<f_0,e_0\>=\<g_0,e_0\>
\ee
in the elliptic and hyperbolic cases, 
\be\label{isom12}
 \<f_0,e_1\>=\<g_0,e_1\>
\ee
in the parabolic case, and
\be\label{isom2}
\<f_0,d_0\>=\<g_0,d_0\>.
\ee

We also know that $A_f$ and $A_g$ share a common eigenbundle $\Delta$ with respect to a common principal curvature $\lambda$, and that $\Delta_1\oplus\Delta_2\subset \Delta$.
Let $N_{f_0}$ and $N_{g_0}$ denote unit vector fields normal to $f_0$ and $g_0$.
It follows from \eqref{A0}, \eqref{princurv1} and $\eqref{princurv2}$ that $f_0$ and $g_0$ share a common principal direction, with common principal curvature $\lambda$ satisfying
\be \label{lambda1}
\lambda=-\frac{\<N_{f_0},e_0\>}{\<f_0,e_0\>}=-\frac{\<N_{g_0},e_0\>}{\<g_0,e_0\>}=-\frac{\<N_{f_0},d_0\>}{\<f_0,d_0\>}=-\frac{\<N_{g_0},d_0\>}{\<g_0,d_0\>}
\ee
in the elliptic and hyperbolic cases, and
\begin{equation} \label{lambda2}
\lambda=-\frac{\<N_{f_0},e_1\>}{\<f_0,e_1\>}=-\frac{\<N_{g_0},e_1\>}{\<g_0,e_1\>}=-\frac{\<N_{f_0},d_0\>}{\<f_0,d_0\>}=-\frac{\<N_{g_0},d_0\>}{\<g_0,d_0\>}
\end{equation}
in the parabolic case. These are the other necessary conditions on the profiles in order to obtain true deformations.

   Therefore, for a complete classification of the truly deformable hypersurfaces of $\mathbb{H}^k\times \mathbb{S}^{n-k+1}$, $2\leq k\leq n-1$, $n\geq 6$, the following problems should be solved:
   
   \begin{itemize}
   \item[(i)] For which conformally ruled hypersurfaces of dimension $n\geq 6$
of $\mathbb{H}^k\times \mathbb{S}^{n-k+1}$, $2\leq k\leq n-1$, can one construct tensors $A$ and $R$, and a one form $S$ satisfying \eqref{derR}, \eqref{derS2}, \eqref{dert}, \eqref{Gauss} and \eqref{codazzi}?

\item[(ii)] Which hypersurfaces $f_0\colon M_0^3\to \mathbb{H}^\ell\times \mathbb{S}^{4-\ell}$, $\ell\in \{1,2,3\}$ and
$f_0\colon M_0^2\to \mathbb{H}^\ell\times \mathbb{S}^{3-\ell}$, $\ell\in \{1,2\}$ admit isometric deformations satisfying the preceding additional conditions?
\end{itemize}

\noindent Ruy Tojeiro\\
Universidade de S\~ao Paulo\\
Instituto de Ci\^encias Matem\'aticas e de Computa\c c\~ao.\\
Av. Trabalhador S\~ao Carlense 400\\
13566-590 -- S\~ao Carlos\\
BRAZIL\\
\texttt{tojeiro@icmc.usp.br}
\bigskip

\noindent Miguel Ibieta Jimenez\\
Universidade Estadual de Campinas\\
Instituto de Matemática, Estatística e de Computação Científica\\
Rua Sérgio Buarque de Holanda, 651\\
13083-859 -- Campinas\\
BRAZIL\\
\texttt{mibieta@unicamp.br}

\end{document}